\documentclass[reqno, 11pt]{amsart}
\usepackage[utf8]{inputenc}
\usepackage{amssymb,amsmath,amsfonts,amsthm,calrsfs,mathtools}
\usepackage[all]{xy}
\usepackage{color}
\usepackage[english]{babel}
\usepackage[T1]{fontenc}
\usepackage{multicol}
\usepackage{latexsym}
\usepackage{enumitem}
\usepackage{tabularx}
\usepackage{mdframed}
\usepackage{mathrsfs}
\usepackage[a4paper,top=3cm,bottom=2cm,left=3cm,right=3cm,marginparwidth=1.75cm]{geometry}
\usepackage[colorinlistoftodos]{todonotes}
\usepackage[colorlinks=true, allcolors=red]{hyperref}
\usepackage{tikz}
\usetikzlibrary{shapes.geometric}
\usetikzlibrary{arrows.meta}
\usepackage{algorithmic}
\usepackage{algorithm}
\graphicspath{{figures/}}
\usepackage{caption}
\usepackage{subcaption}
\usepackage{xcolor}
\usepackage{booktabs}
\theoremstyle{plain}
\newtheorem{theorem}{Theorem}

\newtheorem{proposition}{Proposition}

\newtheorem{corollary}{Corollary}

\theoremstyle{definition}

\newtheorem{definition}{Definition}

\newtheorem{remark}{Remark}

\providecommand{\keywords}[1]

\title{Topological Detection of Hopf Bifurcations via Persistent Homology: A Functional Criterion from Time Series}
\author{Jhonathan Barrios} 
\address{Centre of Mathematics (CMAT),University of Minho, Portugal }
\email{id10605@uminho.pt}
\thanks{Corresponding author}
\author{Y\'asser Ech\'avez}
\address{Department of Mathematics, Universidad del Atl\'{a}ntico, Colombia}
\email[(Yásser Echávez)]{yjechavez@mail.uniatlantico.edu.co}

\author{Carlos F. \'{A}lvarez}
\email[(Carlos F. Álvarez)]{cfalvarez@mail.uniatlantico.edu.co}

\makeatletter

\date{\today}

\begin{document}

\begin{abstract}
We propose a topological framework for detecting Hopf-type dynamical transitions directly from scalar time series. The method combines delay-coordinate reconstruction with persistent homology and uses the maximum persistence of one-dimensional homology classes as a scalar descriptor of cyclic structure. For the supercritical Hopf setting, we derive finite-resolution persistence bounds that relate detectability of the reconstructed periodic orbit to its geometry, sampling quality, and finite-data perturbations. A derivative-based estimator is then introduced to localize the critical parameter from the sampled topological functional. The approach is evaluated on the Hopf normal form, the Lorenz system, and a reduced Belousov--Zhabotinsky model. The numerical experiments show accurate finite-resolution localization of the corresponding transitions and illustrate the effects of embedding parameters, temporal sampling, smoothing, observational noise, and transient removal. These results support persistent homology as an interpretable data-driven tool for detecting geometric reorganizations in nonlinear time series.
\end{abstract}

\maketitle
\markright{HOPF'S TOPOLOGICAL CRITERION}

\section{Introduction}

Bifurcations constitute one of the fundamental mechanisms through which nonlinear dynamical systems undergo qualitative changes in their behavior \cite{Hassard1981}. Among them, the Hopf bifurcation occupies a central place: as a pair of complex-conjugate eigenvalues crosses the imaginary axis, a periodic orbit is locally created or destroyed under the usual nondegeneracy conditions \cite{Hassard1981,Perko2013,Strogatz2018}. In the supercritical case, which is the principal theoretical setting of this work, a stable equilibrium loses stability and an attracting periodic orbit emerges. Hopf bifurcations provide a classical mechanism for the onset of self-sustained oscillations in physical, biological, chemical, and engineering systems, including circadian rhythms, neuronal dynamics, oscillatory chemical reactions, and stability problems in control and
electrical systems \cite{GonzeGoldbeter2006,Izhikevich2007}. Consequently, localizing the transition between equilibrium and oscillatory regimes from observed data is relevant for understanding and characterizing qualitative changes in nonlinear dynamics \cite{Kuznetsov2004,Strogatz2018}.

Classical bifurcation analysis relies on tools such as Jacobian eigenvalues, normal forms, bifurcation diagrams, phase portraits, and Lyapunov exponents \cite{Barreira2017,Kuznetsov2004,Perko2013}. Local spectral and  normal-form methods typically require knowledge of the governing equations or sufficient structural information about the underlying system. Although dynamical indicators such as Lyapunov exponents can also be estimated from data, their numerical estimation may depend strongly on sampling, series length, reconstruction choices, and measurement noise \cite{BradleyKantz2015}. These considerations motivate complementary data-driven approaches for detecting qualitative dynamical changes when only observed time series are available.

Takens' theorem provides a mathematical foundation for studying dynamical systems from scalar observations. Under generic conditions, the underlying attractor can be reconstructed by a delay-coordinate map constructed from a single scalar time series \cite{SauerYorkeCasdagli1991,Takens1981}. This makes it possible to study geometric and topological properties of the dynamics without direct access to the governing equations.

Topological Data Analysis (TDA), and persistent homology in particular, offers a multiscale description of such reconstructed geometry. Persistent homology tracks the birth and death of topological features, including connected components, cycles, and higher-dimensional structures, as a scale parameter varies \cite{ChazalMichel2021,EdelsbrunnerHarer2010,Ghrist2008}. Persistence therefore provides a natural way to quantify the prominence and stability of geometric structures in reconstructed time series without identifying short-lived features a priori as measurement noise \cite{Perea2015}.

Persistence-based approaches to dynamical transitions have already been developed in several forms. Berwald and Gidea \cite{BerwaldGidea2014} used persistent homology of delay-embedded time series to detect critical transitions and introduced a quantitative indicator based on Wasserstein distances between persistence diagrams from nearby windows. Subsequent approaches have employed CROCKER plots \cite{Guzel2022}, zigzag persistent homology \cite{Tymochko2020}, and other topological approaches to dynamical and stochastic bifurcation detection \cite{Maletic2016,Tanweer2024}. Thus, the present work is not the first quantitative use of persistent homology for transition detection. Its distinction lies in the particular scalar observable considered and in the accompanying geometric analysis: we study the maximum one-dimensional persistence of each delay reconstruction and derive finite-resolution conditions connecting its detectability to the geometry of the reconstructed periodic orbit, sampling density, and finite-data perturbations.

Hopf bifurcation detection from time series has also been approached through non-topological methods. Classical early-warning indicators, including variance, autocorrelation, and quantities associated with critical slowing down, require no supervised training data but may not uniquely identify the underlying transition mechanism \cite{BoettigerHastings2012,Scheffer2009}. More recently, data-driven and deep-learning methods have been developed to classify or anticipate bifurcations from simulated, raw, or delay-embedded time series \cite{Bury2021, Park2026}. Such approaches address a complementary problem and generally provide a different form of interpretability from the geometric criterion studied here.

Against this background, we develop a topological framework centered on the dominant one-dimensional persistence functional. In the supercritical attracting-cycle setting, the transition from an equilibrium to a periodic orbit changes the reconstructed geometry from a point-like set to a cyclic set. We establish pointwise persistence bounds for this transition and a finite-resolution separation criterion in terms of reach, sampling density, and finite-data perturbations. A finite-sample Gaussian bound connects the deterministic stability result to the observational-noise model used numerically. We then introduce a derivative-based estimator that localizes the strongest change of the sampled persistence functional along an ordered parameter sweep. The theoretical detectability result and the numerical localization procedure are deliberately distinguished: the former applies to the supercritical attracting-cycle regime, whereas the latter can also be used descriptively outside that setting.

The numerical study reflects this distinction. The Hopf normal form provides the canonical supercritical case with a known critical parameter and is used to examine the predicted square-root onset and finite-resolution effects. The classical Lorenz system is included deliberately as an empirical stress test: its Hopf bifurcation is subcritical and therefore lies outside the attracting-periodic-orbit hypotheses of the main theorem. A reduced Belousov--Zhabotinsky model
provides a complementary supercritical configuration in which the attracting oscillatory branch lies on the opposite side of the chosen parameter orientation. Together, these experiments assess localization, sensitivity to numerical and reconstruction choices, observational noise, and near-critical transient effects.



\section{Mathematical Background}\label{sec:background}

This section fixes notation and recalls the facts used in Section~\ref{sec:criterion}, at the level of detail needed by a reader who knows one of the two source fields well and the other only in passing. Readers fluent in bifurcation theory can skip to Section~\ref{subsec:homology}; readers fluent in persistent homology can skip to Section~\ref{sec:criterion}.

\subsection{Hopf bifurcations}
Although the notion of dynamical system is broad in mathematics, throughout this work we restrict attention to dynamical systems generated by the solution (flow) of a system of ordinary differential equations. A \textit{bifurcation} of a dynamical system is a qualitative change in its dynamics produced by varying a parameter. Consider a one-parameter family of smooth dynamical systems
\begin{equation}\label{eq:system}
  \dot{x} = f(x,\mu), \qquad x \in \mathbb{R}^n, \quad \mu \in I \subset \mathbb{R},
\end{equation}
where $\mu$ is a control parameter. An \textit{equilibrium point} $x^{\ast}$ satisfies $f(x^{\ast},\mu)=0$. Recall that the phase space of system \eqref{eq:system} is $\mathbb{R}^n$, the space of all possible states $x$. For fixed $\mu_{0}\in I$, a closed orbit $\Gamma_{\mu_{0}}$ is called a \textit{limit cycle} if there exists $\varepsilon>0$ such that no other closed orbit lies within distance $\varepsilon$ of $\Gamma_{\mu_{0}}$.

A \textit{Hopf bifurcation} occurs when a pair of complex conjugate eigenvalues $\lambda(\mu)$ of the Jacobian $Df(x^{\ast},\mu)$ crosses the imaginary axis as $\mu$
varies. Under suitable non-degeneracy conditions, this transition leads to the emergence of a limit cycle in a neighborhood of the equilibrium \cite{Kuznetsov2004,Strogatz2018}. More precisely, there exists a critical value $\mu_{c}$ such that
\[  \operatorname{Re}(\lambda(\mu_{c}))=0.\] For parameter values near $\mu_{c}$, the system transitions between a stable equilibrium and periodic dynamics. This phenomenon describes the emergence of self-sustained oscillations in numerous physical, biological, and chemical systems.

Concretely: near $(x^*,\mu_c)$, after a smooth change of coordinates the flow reduces to a two-dimensional system which, in polar coordinates $(r,\theta)$ centered at $x^*$, takes the form
\[\dot r = \mu r + l_1 r^3 + O(r^5), \qquad \dot\theta = \omega + O(r^2),\]
where $l_1$ is the \emph{first Lyapunov coefficient} and $\omega=\operatorname{Im}\lambda(\mu_c)\neq0$ \cite{Kuznetsov2004}. Truncating at cubic order gives the normal form used in the numerical experiments below (Case A, Section~\ref{sec:experiments}). For $l_1<0$, $r=0$ is stable for $\mu<0$, loses stability at $\mu=0$, and a stable circle of radius $r=\sqrt{-\mu/l_1}$ appears for $\mu>0$: the reconstructed point cloud switches from concentrating near a single point to concentrating near a circle whose radius grows from zero. Tracking this radius growth topologically is precisely the goal of Section~\ref{sec:criterion}.

\begin{remark}[Supercritical versus subcritical Hopf bifurcations]
After translating the critical parameter so that $\nu=\mu-\mu_c$, the local radial normal form is

\[
\dot r=\nu r+l_1r^3+O(r^5).
\]

If $l_1<0$, the Hopf bifurcation is supercritical, a stable periodic orbit exists on the side $\nu>0$, and its radius tends continuously to zero as $\nu\downarrow0$. This is the setting considered in Theorem \ref{thm:hopf}, because post-transient trajectories sample the attracting periodic orbit and therefore produce a cyclic reconstructed point cloud.

If $l_1>0$, the Hopf bifurcation is subcritical. The small periodic branch exists on the opposite side of the critical parameter and is unstable; its radius also tends to zero as the bifurcation is approached. The difficulty for the present data-driven framework is therefore not a discontinuous finite-size collapse of the local periodic orbit, but its instability: generic post-transient time series do not converge to this branch and hence do not directly sample it as an attracting set. The theoretical result below is consequently restricted to the supercritical attracting-cycle case. Subcritical transitions may nevertheless produce detectable changes in reconstructed geometry, but such behavior lies outside the guarantee of Theorem \ref{thm:hopf} and is treated only empirically in the experiments.

\end{remark}

\subsection{Phase space reconstruction}
\label{subsec:reconstruction}

In many applications the underlying dynamical system is not explicitly known and only a scalar time series is observed. Examples arise in gait dynamics, biomechanical sensors, and financial time series, where the system must be inferred from data \cite{Barrios2026,Gidea2018}. Here, an \textit{attractor} is a compact invariant set toward which all nearby orbits converge asymptotically. Takens' theorem states that, under generic conditions, the geometry of the underlying attractor can be reconstructed by means of a delay embedding constructed from the observed time series \cite{Takens1981}.

Let $x(t)$ be a scalar time series. The \textit{delay embedding} is the map
\begin{equation}\label{eq:embedding}
\Phi(t) =\bigl(x(t),\, x(t-\tau),\, x(t-2\tau),\, \dots,\, x(t-(m-1)\tau)\bigr)  \in \mathbb{R}^{m},
\end{equation}
where $\tau>0$ is the time delay and $m\geq 1$ is the embedding dimension. Under generic conditions, $\Phi$ defines an embedding of the original attractor into $\mathbb{R}^{m}$ that preserves its topological and geometric properties \cite{Takens1981}; this result was later extended to attractors of fractal dimension in \cite{SauerYorkeCasdagli1991}. Intuitively, $\Phi$ works because each delayed coordinate $x(t-k\tau)$ is a different, generically nonlinear function of the full state at time $t$ (namely, the state evolved backward by $k\tau$ and re-observed through the same scalar sensor); stacking $m\geq 2n+1$ such coordinates supplies enough independent views of the trajectory to separate points that a single scalar observation would confuse, much as several projections of a curve resolve self-intersections that any one projection creates. This heuristic also explains a known practical failure mode, revisited in Remark~\ref{rem:delay_dependence}: if $\tau$ is too small relative to the sampling rate, consecutive coordinates $x(t-k\tau)$ are nearly equal, the $m$ ``different views'' collapse onto one another, and $\Phi$ flattens the attractor toward the diagonal $x_1=\cdots=x_m\subset\mathbb{R}^m$ instead of unfolding it. The resulting point cloud

\begin{equation}\label{eq:cloud}
  X = \{\Phi(t_{k})\}_{k=1}^{N} \subset \mathbb{R}^{m}
\end{equation}
thus constitutes a reconstruction of the attractor from the observed time series.

\subsection{Persistent homology}\label{subsec:homology}

For a reader whose background is in dynamical systems rather than computational topology, it is most useful to see concretely what happens to a point cloud sampled near a circle, since this is exactly the situation this paper analyzes (Section~\ref{sec:criterion}).

Recall that a \textit{simplicial complex} is a finite collection of simplices (vertices, edges, triangles, and their higher-dimensional analogues) closed under taking faces. Given a finite point cloud $X\subset\mathbb{R}^{m}$ and a scale parameter $\varepsilon\geq 0$, the \textit{Vietoris--Rips complex} (VR) $\mathrm{VR}(X,\varepsilon)$ is the simplicial complex whose simplices are subsets of $X$ with pairwise distances at most $\varepsilon$. Varying $\varepsilon$ yields a nested family $\{\mathrm{VR}(X,\varepsilon)\}_{\varepsilon\geq 0}$, called a \textit{filtration}.

Suppose
\begin{equation*}
   X=\{x_0,\ldots,x_{N-1}\}, \qquad N\geq4,
\end{equation*}

consists of $N$ equally spaced points on a circle of radius $r$. Let

\begin{equation*}
  \ell_1 = 2r\sin\left(\frac{\pi}{N}\right)
\end{equation*}

denote the Euclidean distance between adjacent samples, and let

\begin{equation*}
 \ell_2 = 2r\sin\left(\frac{2\pi}{N}\right)
\end{equation*}

denote the distance between second neighbors. For $0\leq\varepsilon<\ell_1$, the Vietoris--Rips complex contains no edges and therefore consists of $N$ isolated vertices. When $\ell_1\leq\varepsilon<\ell_2$, precisely adjacent vertices are connected. The resulting
Vietoris--Rips complex is the cycle graph $C_N$, and therefore

\begin{equation*}
   H_1(VR(X,\varepsilon);\mathbb F_2) \cong \mathbb F_2.
\end{equation*}

At larger scales, additional chords and higher-dimensional simplices enter the filtration and eventually fill the loop. In particular, once $\varepsilon\geq\operatorname{diam}(X)$, every pair of vertices is connected, the Vietoris--Rips complex is the full
simplex and its first homology is trivial.

This elementary example illustrates the two geometric scales underlying persistence: a local scale at which neighboring observations connect and a larger scale associated with the global cyclic geometry. Section~\ref{sec:topological_criterion} replaces this elementary picture by quantitative bounds involving the sampling density and the reach of the reconstructed periodic orbit.

Following Hatcher \cite[Chapter~2]{Hatcher2002}, the first homology group $H_{1}$ of a topological space is the abelianization of its fundamental group. For the purposes of persistent homology, $H_1$ is computed with coefficients in a field (throughout this work, $\mathbb{Z}/2\mathbb{Z}$), so that each $H_1(\mathrm{VR}(X,\varepsilon))$ is a vector space and the persistence module $\{H_1(\mathrm{VR}(X,\varepsilon))\}_{\varepsilon\geq0}$ admits the interval-decomposition on which the persistence diagram is based; this is what makes birth and death scales well defined. In the circle example above, $H_1(C_N;\mathbb F_2)\cong\mathbb F_2$, generated by the class represented
by the loop traversing the cycle graph once.

Persistent homology tracks the birth and death of homology classes as $\varepsilon$ varies across the filtration, over all dimensions and all classes at once (not only the single dominant one of the circle example). The output is a \textit{persistence diagram}
\[D = \{(b_{i}, d_{i})\},\]
where $b_{i}$ and $d_{i}$ denote, respectively, the birth and death scale of the $i$-th homological feature, and the quantity $d_{i}-b_{i}$ measures its \textit{persistence} \cite{Wasserman2018}. Applied to dynamical systems, non-trivial classes in $H_{1}$ are typically associated with periodic orbits or limit cycles in the reconstructed attractor \cite{Perea2019}, exactly as in the example above; classes with very short persistence, $d_i\approx b_i$, are less robust to perturbations and may arise from sampling variability or small-scale geometric structure rather than from a coherent large-scale loop.

The maximum persistence functional used in this work is not itself new: an analogous quantity, computed from delay-embedded persistence diagrams, was used by \cite{Emrani2014} to detect harmonic structure in respiratory sound signals for wheeze detection, and the same quantity is known in Hamiltonian Floer theory as the \emph{boundary depth} of a persistence module \cite{Polterovich2020}. Our contribution is not the functional itself but the threshold-crossing guarantee of Theorem~\ref{thm:hopf} that accompanies it in the Hopf bifurcation setting.

For a finite point cloud $Y\subset\mathbb R^m$ and a scale $\varepsilon\ge0$, let
\begin{equation*}
\beta_i(\varepsilon;Y):=\dim_{\mathbb F_2}H_i\big(\mathrm{VR}(Y,\varepsilon);\mathbb F_2\big)
\end{equation*}
denote the $i$-th Betti number of the Vietoris--Rips complex at scale $\varepsilon$. This quantity counts the independent $i$-dimensional homology classes present at scale $\varepsilon$ and underlies the aggregate descriptor $\|\beta_1\|_{L^1}$ used in Section~\ref{sec:methodology}.

Two compact sets $A,B\subset\mathbb R^m$ are compared through their Hausdorff distance
\begin{equation}\label{eq:hausdorff_def}
d_H(A,B):=\max\Big\{\sup_{a\in A}\inf_{b\in B}\|a-b\|,\ \sup_{b\in B}\inf_{a\in A}\|a-b\|\Big\}.
\end{equation}
Hypotheses \textnormal{(H2)} and \textnormal{(H3)} in Section~\ref{sec:topological_criterion} are stated in terms of $d_H$. Two abstract metric spaces $X$ and $Y$ are compared through their Gromov--Hausdorff distance $d_{GH}(X,Y)$, defined as the infimum of the Hausdorff distance between the images of $X$ and $Y$ under isometric embeddings of both spaces into a common metric space $Z$. When $A,B\subset\mathbb R^m$ already share the Euclidean ambient space, the identity embedding gives $d_{GH}(A,B)\le d_H(A,B)$. Two persistence diagrams $D$ and $D'$ are compared through their bottleneck distance
\begin{equation}\label{eq:bottleneck_def}
d_B(D,D'):=\inf_{\gamma}\sup_{p\in D\cup\Delta}\|p-\gamma(p)\|_\infty,
\end{equation}
where $\Delta$ denotes the diagonal $\{(t,t)\mid t\in\mathbb R\}$ counted with infinite multiplicity and the infimum ranges over bijections $\gamma$ from $D\cup\Delta$ to $D'\cup\Delta$. The distances $d_H$, $d_{GH}$, and $d_B$ are used respectively in hypotheses \textnormal{(H2)} and \textnormal{(H3)} of Section~\ref{sec:topological_criterion} and in the proof of Theorem~\ref{thm:hopf}.

\section{Hopf's Topological Criterion}\label{sec:criterion}

The emergence of a limit cycle after a Hopf bifurcation implies a qualitative change in the geometry of the system's attractor. Before the bifurcation, the trajectories converge toward an equilibrium point, resulting in a point cloud concentrated in a small region of the reconstructed space. After the bifurcation, the dynamics evolve toward a periodic orbit, generating an approximately circular structure in phase space. From a topological perspective, this transition corresponds to the emergence of a nontrivial homology class in dimension one \cite{Ghrist2008,Perea2019}. While no significant persistent cycles are observed before the bifurcation, a dominant class in $H_1$ appears afterward, associated with the presence of the limit cycle. This observation suggests that the Hopf transition can be detected through the emergence of a persistent dominant class in $H_1$. To quantify this phenomenon, we propose a criterion based on persistent homology for detecting Hopf bifurcations directly from observed time series.

\subsection{Setup}

Consider the one-parameter family \eqref{eq:system}. For each $\mu$, let $x(t;\mu)$ be an observed scalar time series. Following the construction \eqref{eq:embedding} of Section~\ref{subsec:reconstruction}, the delay embedding
\begin{equation*}
  \Phi_\mu(t)=\big(x(t;\mu),x(t-\tau;\mu),\ldots,x(t-(m-1)\tau;\mu)\big)\in\mathbb R^m
\end{equation*}
yields, as in \eqref{eq:cloud}, the reconstructed point cloud
\begin{equation*}
X_{\mu}=\{\Phi_{\mu}(t_{k})\}_{k=1}^{N_{\mu}}\subset\mathbb{R}^{m}.
\end{equation*}

Applying the construction of Section~\ref{subsec:homology} to $X_{\mu}$ gives the one-dimensional persistence diagram

\begin{equation}\label{eq:persistence_diagram_mu}
D^{(1)}_{\mu} = \{(b_{i,\mu},\,d_{i,\mu})\}_{i\in I_\mu},
\end{equation}

\begin{remark}[Embedding-parameter selection and comparability across the parameter sweep]
\label{rem:embedding_parameters}
The embedding parameters $(\tau,m)$ are kept fixed throughout each parameter sweep. They may be selected from a representative time series using standard data-driven criteria, such as mutual information for the delay \cite{Fraser1986} and False Nearest Neighbors for the embedding dimension \cite{kennel1992}. Re-estimating $(\tau,m)$ independently at every parameter value would change the reconstruction map itself and could therefore introduce artificial variations in $\mathcal H(\mu)$ that are unrelated to the underlying dynamical transition. Sensitivity to the embedding parameters is instead assessed by repeating the complete parameter sweep for several fixed choices of $(\tau,m)$, as in Figure ~\ref{fig:embedding_sensitivity}. A delay that is too small can flatten the reconstructed orbit toward the diagonal $x_1=\cdots=x_m$, whereas an excessively large delay may reduce the coherence of the reconstructed geometry.
\end{remark}

\subsection{Dominant topological functional}

\begin{definition}[Dominant topological functional]\label{def:dominant_functional}
For a finite point cloud $Y$, let $D^{(1)}(Y)$ denote the persistence diagram in dimension one obtained from the Vietoris--Rips filtration of $Y$. We define
\begin{equation}
 \mathscr H(Y):=\max_{(b,d)\in D^{(1)}(Y)}(d-b),
\end{equation}

with the convention $\mathscr H(Y)=0$ when $D^{(1)}(Y)$ has no off-diagonal points. For the reconstructed point cloud $X_\mu$, we write
\[ \mathcal H(\mu):=\mathscr H(X_\mu). \]
Equivalently, $\mathcal H(\mu)=\max_{(b,d)\in D^{(1)}_\mu}(d-b)$, where $D^{(1)}_\mu=D^{(1)}(X_\mu)$ is the persistence diagram introduced in \eqref{eq:persistence_diagram_mu}. If $\widetilde X_\mu$ denotes a perturbed or noisy reconstruction, we analogously write
\[ \widetilde{\mathcal H}(\mu):=\mathscr H(\widetilde X_\mu). \]
The functional measures the persistence of the most prominent one-dimensional homological feature of the reconstructed data.
\end{definition}

\subsection{Topological detection criterion}
\label{sec:topological_criterion}

We now differentiate the geometric attractor set from the finite reconstructed point cloud on which persistent homology is calculated. This distinction separates the underlying attractor geometry and its sampling density from finite sampling, residual transient, and measurement perturbations.

\medskip
\noindent
\textbf{(H1) (Supercritical attracting-set regime).}

The system undergoes a supercritical Hopf bifurcation at $\mu=\mu_c$ \cite{Kuznetsov2004}. We orient the parameter so that the stable periodic branch exists for $\mu>\mu_c$. In a neighborhood of $\mu_c$, let

\begin{equation*}
A_\mu=
   \begin{cases}
       \{x^*(\mu)\}, & \mu<\mu_c,\\[2mm]
       \Gamma_\mu, & \mu>\mu_c,
   \end{cases}
\end{equation*}

where $x^*(\mu)$ is the stable equilibrium and $\Gamma_\mu$ is the stable periodic orbit. The delay-coordinate map $\Phi_\mu$, with common embedding parameters $(\tau,m)$ throughout the parameter sweep, is assumed to be an embedding on the relevant attracting set in the sense of Takens' theorem \cite{Takens1981}. We write

\begin{equation*}
   M_\mu:=\Phi_\mu(A_\mu),
\end{equation*}

where $\Phi_\mu(A_\mu)$ denotes the image of the attracting set under the corresponding delay-coordinate map.

\medskip
\noindent
\textbf{(H2) (Sampling density on the periodic side).}

For $\mu>\mu_c$, let $P_\mu\subset M_\mu$ be a finite reference sample satisfying

\begin{equation*}
    d_H(P_\mu,M_\mu)\leq\delta_\mu,
\end{equation*}

where $d_H$ denotes the Hausdorff distance \eqref{eq:hausdorff_def}. Let

\begin{equation*}
   R_\mu:=\operatorname{rch}(M_\mu)>0.
\end{equation*}

Following Federer \cite{Federer1959}, the reach $\operatorname{rch}(K)$ of a compact set $K\subset\mathbb{R}^m$ is the supremum of the radii $r>0$ such that every point at distance less than $r$ from $K$ has a unique nearest
point in $K$. We assume

\begin{equation} \label{eq:sampling_density}
  \delta_\mu< \frac{1}{2}\sqrt{\frac{3}{10}}\,R_\mu.
\end{equation}

For $\mu<\mu_c$, we take $P_\mu=M_\mu$, which consists of a single point.

\medskip
\noindent
\textbf{(H3) (Finite-data perturbation).}

Let $\widetilde X_\mu$ denote the point cloud actually reconstructed from the finite observed time series, after transient removal and possible measurement perturbation. We assume

\begin{equation}\label{eq:finite_data_perturbation}
   d_H(P_\mu,\widetilde X_\mu)\leq q_\mu.
\end{equation}

If $a_\mu$ bounds the discrepancy between $P_\mu$ and the corresponding finite noise-free reconstruction, and the scalar observational noise satisfies

\begin{equation*}
   \|\xi\|_\infty\leq\sigma_\mu,
\end{equation*}

then the triangle inequality and the structure of the delay embedding give

\begin{equation}\label{eq:q_noise_bound}
 q_\mu \leq a_\mu+\sqrt{m}\,\sigma_\mu,
\end{equation}

because each of the $m$ coordinates of a delay vector is perturbed in absolute value by at most $\sigma_\mu$.

All homology groups below are computed with coefficients in the field $\mathbb{F}_2=\mathbb{Z}/2\mathbb{Z}$.

\begin{theorem}[Hopf topological criterion: pointwise persistence bounds]
\label{thm:hopf}

Suppose that hypotheses \textnormal{(H1)--(H3)} hold. Then:

\begin{enumerate}
    \item[(i)] for $\mu<\mu_c$,

    \begin{equation} \label{eq:pre_hopf_bound}
      \widetilde{\mathcal H}(\mu)\leq4q_\mu;
    \end{equation}

    \item[(ii)] for $\mu>\mu_c$, define the geometric persistence margin

    \begin{equation} \label{eq:geometric_margin}
       \eta_{\mu}^{\mathrm{geom}} := 2\left(\sqrt{\frac{3}{10}}\,R_\mu-2\delta_\mu\right)>0.      
    \end{equation}

    Then

\begin{equation}\label{eq:post_hopf_bound}
\widetilde{\mathcal H}(\mu) \geq \eta_{\mu}^{\mathrm{geom}}-4q_\mu. 
\end{equation}

In particular, whenever

\begin{equation}\label{eq:detectability_condition}
\eta_{\mu}^{\mathrm{geom}}>4q_\mu,
\end{equation}
the observed persistence diagram contains a nontrivial one-dimensional class whose persistence is bounded away from zero.
\end{enumerate}
\end{theorem}

\begin{proof}
We use the same scale convention as in Section \ref{subsec:homology}: $VR(Y,\varepsilon)$ contains a simplex whenever all pairwise distances between its vertices are at most $\varepsilon$, while $\check{C}(Y,r)$ denotes the \v{C}ech complex associated with the cover of $Y$ by Euclidean balls of radius $r$.

With these conventions, the Euclidean \v{C}ech--Vietoris--Rips comparison gives
\begin{equation}\label{eq:cech_rips_interleaving}
  \check{C}(Y,r) \subseteq VR(Y,2r) \subseteq \check{C}(Y,\sqrt{2}\,r).
\end{equation}

Indeed, the first inclusion is immediate, while the second follows from the Euclidean Rips--\v{C}ech comparison of de Silva and Ghrist \cite[Theorem~2.5]{deSilvaGhrist2007}. Their dimension-dependent constant is bounded above by $\sqrt{2}$, which gives the dimension-independent form used in \eqref{eq:cech_rips_interleaving}.

We first consider the case $\mu<\mu_c$. By \textnormal{(H1)}, $P_\mu=M_\mu$ consists of a single point. Therefore its one-dimensional persistence diagram contains no off-diagonal points and

\begin{equation*}
  \mathscr H(P_\mu)=0.
\end{equation*}

Since $P_\mu,\widetilde X_\mu\subset\mathbb R^m$ lie in a common ambient metric space, the Gromov--Hausdorff distance is bounded by the Hausdorff distance \eqref{eq:hausdorff_def}, $d_{GH}(P_\mu,\widetilde X_\mu)\le d_H(P_\mu,\widetilde X_\mu)$. Hypothesis \textnormal{(H3)}, that is \eqref{eq:finite_data_perturbation}, gives $d_H(P_\mu,\widetilde X_\mu)\le q_\mu$, hence $d_{GH}(P_\mu,\widetilde X_\mu)\le q_\mu$. The Gromov--Hausdorff stability theorem for Vietoris--Rips persistence \cite[Theorem~3.1]{Chazal2009}, translated to the scale convention used here, then bounds the bottleneck distance $d_B$ \eqref{eq:bottleneck_def} between the two diagrams as

\begin{equation}  \label{eq:gh_stability}
  d_B\!\left(D^{(1)}(P_\mu), D^{(1)}(\widetilde X_\mu) \right) \leq 2d_{GH}(P_\mu,\widetilde X_\mu) \leq 2q_\mu.
\end{equation}

The factor $2$ in \eqref{eq:gh_stability} comes from the difference in scale convention. Chazal et al.\ \cite{Chazal2009} parametrize their Rips filtration so that simplices have diameter less than twice the filtration parameter, whereas here the filtration parameter $\varepsilon$ is the diameter threshold itself.

If two persistence diagrams are within bottleneck distance $\varepsilon$, then the persistence $d-b$ of two matched off-diagonal points can differ by at most $2\varepsilon$. If a point is matched to the diagonal, its persistence is at most $2\varepsilon$. Consequently,

\begin{equation*}
  \left| \mathscr H(P_\mu)- \widetilde{\mathcal H}(\mu) \right| \leq4q_\mu.
\end{equation*}

Since $\mathscr H(P_\mu)=0$, this proves \eqref{eq:pre_hopf_bound}.

We now consider $\mu>\mu_c$. Then

\begin{equation*}
    M_\mu=\Phi_\mu(\Gamma_\mu)
\end{equation*}

is an embedded simple closed curve. Therefore

\begin{equation*}
   H_1(M_\mu;\mathbb F_2)\cong\mathbb F_2.
\end{equation*}

By \textnormal{(H2)}, $P_\mu$ is a $\delta_\mu$-sample of $M_\mu$. The reconstruction theorem of Niyogi, Smale and Weinberger \cite[Proposition~3.1]{Niyogi2008} implies that, for every radius $r$ satisfying

\begin{equation}\label{eq:reconstruction_interval}
2\delta_\mu<r<\sqrt{\frac{3}{5}}\,R_\mu, 
\end{equation}
the union 
\begin{equation*}
U_\mu(r) := \bigcup_{p\in P_\mu}B(p,r),
\end{equation*}
deformation retracts onto $M_\mu$.

Since Euclidean balls and all their nonempty finite intersections are convex, the covers defining $U_\mu(r)$ are good covers. Moreover, these covers are nested as $r$ increases. The Persistent Nerve Lemma \cite[Lemma~3.4]{ChazalOudot2008} therefore identifies the homology maps of the union-of-balls filtration with those of the associated \v{C}ech filtration, compatibly with the inclusion maps between scales.

More explicitly, if $r\leq r'$ both belong to the interval \eqref{eq:reconstruction_interval}, then $M_\mu\subset U_\mu(r)\subset U_\mu(r')$, and both unions deformation retract onto the same $M_\mu$. Hence the inclusion

\begin{equation*}
  U_\mu(r)\hookrightarrow U_\mu(r')
\end{equation*}

induces an isomorphism on $H_1$. By the Persistent Nerve Lemma, the corresponding inclusion of \v{C}ech complexes induces the same persistent homology map.

By \eqref{eq:sampling_density}, $2\delta_\mu<\sqrt{3/10}\,R_\mu$, so we may choose

\begin{equation*}
   2\delta_\mu<r_1<r_2< \sqrt{\frac{3}{10}}\,R_\mu.
\end{equation*}

Then both $r_1$ and $\sqrt{2}\,r_2$ belong to the reconstruction interval \eqref{eq:reconstruction_interval}, because

\begin{equation*}
 \sqrt{2}\,r_2 < \sqrt{\frac{3}{5}}\,R_\mu.
\end{equation*}

Using \eqref{eq:cech_rips_interleaving}, consider the sequence of inclusions

\begin{equation}\label{eq:cech_rips_sequence}
  \check{C}(P_\mu,r_1) \longrightarrow VR(P_\mu,2r_1) \longrightarrow VR(P_\mu,2r_2) \longrightarrow \check{C}(P_\mu,\sqrt{2}\,r_2).
\end{equation}

The composite map in \eqref{eq:cech_rips_sequence} coincides with the canonical inclusion from $\check{C}(P_\mu,r_1)$ to $\check{C}(P_\mu,\sqrt{2}\,r_2)$. By the reconstruction result and the Persistent Nerve Lemma, this inclusion induces an isomorphism on $H_1(-;\mathbb F_2)$ and is therefore nonzero.

Consequently, the middle map

\begin{equation*}
   H_1\!\left(VR(P_\mu,2r_1);\mathbb F_2\right)\longrightarrow H_1\!\left(VR(P_\mu,2r_2);\mathbb F_2\right)
\end{equation*}

must also be nonzero. Hence the Vietoris--Rips persistence barcode of $P_\mu$ contains an interval whose length is at least

\begin{equation*}
  2(r_2-r_1).
\end{equation*}

Letting $r_1\downarrow2\delta_\mu$ and $r_2\uparrow\sqrt{3/10}\,R_\mu$ gives

\begin{equation} \label{eq:reference_persistence_bound}
  \mathscr H(P_\mu) \geq 2\left(\sqrt{\frac{3}{10}}\,R_\mu-2\delta_\mu \right) = \eta_\mu^{\mathrm{geom}}.
\end{equation}

Finally, hypothesis \textnormal{(H3)} (\eqref{eq:finite_data_perturbation}) and the same Gromov--Hausdorff stability argument used in \eqref{eq:gh_stability} imply

\begin{equation*}
 \left| \widetilde{\mathcal H}(\mu) - \mathscr H(P_\mu) \right| \leq4q_\mu.
\end{equation*}

Combining this inequality with \eqref{eq:reference_persistence_bound} yields

\begin{equation*}
\widetilde{\mathcal H}(\mu) \geq \eta_\mu^{\mathrm{geom}}-4q_\mu,
\end{equation*}

which proves \eqref{eq:post_hopf_bound}.
\end{proof}

\begin{corollary}[Finite-resolution Hopf topological criterion]
\label{cor:finite_resolution}

The pointwise detectability condition \eqref{eq:detectability_condition} need not hold uniformly as $\mu\to\mu_c$. The following statement isolates parameter ranges on which it does.

Fix $0<\Delta<\rho$ and consider the parameter intervals

\begin{equation*}
  K_-=[\mu_c-\rho,\mu_c-\Delta], \qquad K_+=[\mu_c+\Delta,\mu_c+\rho].
\end{equation*}

Define $B_-:= \sup_{\mu\in K_-}4q_\mu$ 
and $B_+:= \inf_{\mu\in K_+} \left[ \eta_\mu^{\mathrm{geom}}-4q_\mu \right].$ 
We assume $B_-<\infty$, which holds in particular whenever $q_\mu$ is bounded on the compact interval $K_-$ (for instance, under the uniform high-probability bound of Corollary~\ref{cor:gaussian_noise}); this boundedness is what allows the threshold $\eta_\Delta$ below to be chosen as a finite real number.

If $B_-<B_+,$ then every threshold satisfying
\begin{equation*}
   \eta_\Delta\in(B_-,B_+)
\end{equation*}
separates the two regimes. More precisely,
\begin{equation*}
 \widetilde{\mathcal H}(\mu)<\eta_\Delta, \qquad \mu\in K_-,
\end{equation*}

whereas

\begin{equation*}
   \widetilde{\mathcal H}(\mu)>\eta_\Delta, \qquad \mu\in K_+.
\end{equation*}

\end{corollary}

\begin{corollary}[Finite-sample Gaussian perturbations]
\label{cor:gaussian_noise}

Suppose the geometric assumptions of Theorem~\ref{thm:hopf} hold and that the observational perturbations are independent Gaussian random variables

\begin{equation*}
 \xi_k\sim\mathcal N(0,s_\mu^2), \qquad k=1,\ldots,n_\mu,
\end{equation*}

where $n_\mu$ denotes the number of scalar observations used to construct the delay vectors. For any $0<\gamma<1$, define

\begin{equation}\label{eq:gaussian_bound}
 t_{\mu,\gamma} = s_\mu \sqrt{2\log\left(\frac{2n_\mu}{\gamma}\right)}.
\end{equation}

Then

\begin{equation}\label{eq:gaussian_probability}
  \mathbb P \left(\max_{1\leq k\leq n_\mu} |\xi_k| \leq t_{\mu,\gamma} \right) \geq 1-\gamma.
\end{equation}

Consequently, if $a_\mu$ denotes the deterministic finite-data contribution to the reconstruction error, then with probability at least $1-\gamma$,

\begin{equation} \label{eq:gaussian_q}
q_\mu \leq a_\mu + \sqrt{m}\,s_\mu\sqrt{ 2\log\left(\frac{2n_\mu}{\gamma}\right)}.
\end{equation}

Thus the conclusions of Theorem~\ref{thm:hopf} hold with the right-hand side of \eqref{eq:gaussian_q} used as a high-probability bound for $q_\mu$.
\end{corollary}

\begin{proof}
For a centered Gaussian variable,

\begin{equation*}
 \mathbb P(|\xi_k|>t) \leq 2\exp\left( -\frac{t^2}{2s_\mu^2} \right).
\end{equation*}

Applying the union bound over the $n_\mu$ scalar observations and choosing $t=t_{\mu,\gamma}$, as given in \eqref{eq:gaussian_bound}, gives \eqref{eq:gaussian_probability}. On this event, each coordinate of every delay vector is perturbed by at most $t_{\mu,\gamma}$, and therefore the Euclidean perturbation of a delay vector is at most $\sqrt{m}\,t_{\mu,\gamma}$. Combining this term with the deterministic finite-data error $a_\mu$ through \eqref{eq:q_noise_bound} yields \eqref{eq:gaussian_q}.
\end{proof}

\begin{remark}[Finite resolution of the criterion]
\label{rem:finite_resolution}

Theorem~\ref{thm:hopf} is pointwise, whereas Corollary~\ref{cor:finite_resolution} provides a uniform separation only on parameter sets bounded away from $\mu_c$. We therefore do not assert the existence of a fixed positive threshold separating the two regimes for parameter values arbitrarily close to the Hopf point.

In a supercritical Hopf bifurcation, the attracting periodic orbit contracts toward the equilibrium as the critical parameter is approached. Consequently, without a corresponding improvement in sampling and perturbation scales, no uniform positive lower bound for the geometric persistence margin can in general be expected arbitrarily close to $\mu_c$. This finite-resolution limitation is intrinsic to the data-driven detection problem rather than an artifact of persistent homology.
\end{remark}

\begin{remark}[Dependence on the delay embedding]
\label{rem:delay_dependence}

The geometric persistence margin in \eqref{eq:geometric_margin} can be written as

\begin{equation*}
\eta_\mu^{\mathrm{geom}} = 2\left(\sqrt{\frac{3}{10}}\, \operatorname{rch}\!\left(\Phi_\mu(\Gamma_\mu)\right) -2\delta_\mu \right).
\end{equation*}

It therefore depends explicitly on the geometry produced by the chosen delay embedding. In particular, a delay that is too small may flatten the reconstructed periodic orbit toward the diagonal $x_1=\cdots=x_m$, reducing its reach and hence the available geometric persistence margin. Conversely, an excessively large delay may degrade the coherence of the reconstructed orbit, as discussed in Remark~\ref{rem:embedding_parameters}.

For the supercritical Hopf normal form, the physical limit-cycle radius scales as $|\mu-\mu_c|^{1/2}$ \cite{Kuznetsov2004}. However, a corresponding asymptotic law for $\operatorname{rch}(\Phi_\mu(\Gamma_\mu))$ under a general delay-coordinate embedding requires additional geometric regularity assumptions. We therefore do not assert such an asymptotic relation in Theorem~\ref{thm:hopf}.
\end{remark}

\begin{remark}[A degenerate case without a bounded attractor]
It is natural to ask whether the criterion of Theorem~\ref{thm:hopf} could yield a false positive on systems for which the linear part alone already exhibits eigenvalues crossing the imaginary axis, without a genuine nonlinear saturation producing a limit cycle. A canonical such
example is the purely linear planar system $\dot z = (i+\mu)z$, $z=x+iy\in\mathbb{C}$, i.e.
\[  \dot x = \mu x - y, \qquad \dot y = x + \mu y,\]
obtained from the Hopf normal form of Section~\ref{sec:experiments} by removing the cubic term $-z|z|^2$. For $\mu<0$ the origin is a stable focus and every orbit spirals into it, so hypothesis~(H1) is satisfied
and $\mathcal{H}(\mu)$ is small, as in the genuine Hopf case. However, for $\mu>0$ the origin becomes an unstable focus and non-equilibrium orbit spirals outward without bound: the system has no compact attracting set for $\mu>0$, and consequently no time series generated by this system can satisfy hypothesis~(H1) for $\mu>0$, since there is no attracting set $A_\mu$ for the trajectory to remain close to. This example therefore lies outside the scope of Theorem~\ref{thm:hopf} rather than contradicting it: on the unstable side there is no compact attracting periodic orbit to which hypotheses \textnormal{(H1)--(H3)} can be applied. In finite observations, such behavior would be expected to produce an expanding or nonstationary reconstruction rather than a point cloud sampling a compact post-bifurcation attractor. Consequently, the example illustrates the importance of the attracting-set assumption in Theorem~\ref{thm:hopf}, rather than providing a counterexample to the proposed criterion.
\end{remark}

\subsection{Topological estimator of the critical parameter}

Let $\{\mu_{j}\}_{j=1}^{M}$ be a discretization of the parameter interval, so that $\mathcal{H}(\mu_j)$ is available only at these sampled points, and only up to whatever estimation noise is present in a persistence computation from a finite time series (Section~\ref{sec:methodology}). Locating the point of maximal change in $\mathcal H$ through a raw finite difference, $|\mathcal{H}(\mu_{j})-\mathcal{H}(\mu_{j-1})|$, is numerically ill-conditioned: differentiation amplifies whatever sampling and estimation noise is already present in $\mathcal H(\mu_j)$, and a single spurious jump away from $\mu_c$ can dominate the maximum. We therefore define the estimator through a smoothed derivative rather than the raw difference.

\begin{definition}[Topological estimator]\label{def:estimator}
Let $\widehat{\mathcal H}$ denote a smoothed estimate of $\mathcal H$ obtained from $\{(\mu_j,\mathcal H(\mu_j))\}_{j=1}^M$ by local polynomial regression (a Savitzky--Golay filter in the numerical implementation of Section~\ref{sec:methodology}). The critical parameter is estimated by
\begin{equation}
  \widehat{\mu}^{\,\mathrm{TDA}}_{c} = \arg\max_{\mu}
  \left|\frac{d}{d\mu}\widehat{\mathcal H}(\mu)\right|.
\end{equation}
Boundary points of the sampled interval are excluded from the search, since numerical derivatives are unreliable at the edge of a finite domain.
\end{definition}

The derivative-based estimator is motivated by the local growth law of the dominant topological feature. At the population level, the following result shows that the singular onset of the derivative is localized at the same critical parameter $\mu_c$ appearing in Theorem~\ref{thm:hopf}. This idealized consistency should be distinguished from finite-sample consistency, since the numerical estimator of Definition~\ref{def:estimator} acts on a discretized and smoothed estimate of $\mathcal H$.

\begin{proposition}[Consistency of the idealized estimator with Theorem~\ref{thm:hopf}]
\label{prop:consistency}

Let $s\in\{-1,+1\}$ indicate the side of the critical parameter on which the attracting periodic orbit exists, and define

\begin{equation*}
    d(\mu):=s(\mu-\mu_c).
\end{equation*}

Suppose that an idealized continuous-parameter topological functional $\mathcal H_\infty$ satisfies

\begin{equation}
 \mathcal H_\infty(\mu)=0, \qquad d(\mu)\leq0,
  \label{eq:idealized_equilibrium_side}
\end{equation}

and, on the periodic side $d(\mu)>0$,

\begin{equation}
  \mathcal H_\infty(\mu) = \kappa\,d(\mu)^\alpha+r(\mu), \qquad \kappa>0, \qquad 0<\alpha<1,
\label{eq:idealized_birth_law}
\end{equation}

where $r$ is continuously differentiable for $d(\mu)>0$ and satisfies

\begin{equation}
 r(\mu) = o\!\left(d(\mu)^\alpha\right), \qquad r'(\mu) = o\!\left(d(\mu)^{\alpha-1}\right) \quad \text{as } d(\mu)\downarrow0.
  \label{eq:birth_law_remainder}
\end{equation}

Then, approaching the critical parameter from the periodic side,

\begin{equation}
 \left| \mathcal H_\infty'(\mu) \right| \sim \alpha\kappa\, d(\mu)^{\alpha-1} \longrightarrow \infty.
\label{eq:derivative_singularity}
\end{equation}

Hence $\mu_c$ is a one-sided singular onset of the derivative of $\mathcal H_\infty$.
\end{proposition}

\begin{proof}

For $d(\mu)>0$, differentiating \eqref{eq:idealized_birth_law} gives

\begin{equation*}
 \mathcal H_\infty'(\mu) = s\alpha\kappa\, d(\mu)^{\alpha-1} + r'(\mu),
\end{equation*}

because $d'(\mu)=s$. By \eqref{eq:birth_law_remainder},

\begin{equation*}
   r'(\mu) = o\!\left(d(\mu)^{\alpha-1}\right),
\end{equation*}

and therefore

\begin{equation*}
\mathcal H_\infty'(\mu) = s\alpha\kappa\, d(\mu)^{\alpha-1} \left(1+o(1)\right).
\end{equation*}

Taking absolute values yields

\begin{equation*}
  \left| \mathcal H_\infty'(\mu) \right| = \alpha\kappa\, d(\mu)^{\alpha-1} \left(1+o(1)\right).
\end{equation*}

Since $0<\alpha<1$, we have $\alpha-1<0$, and hence

\begin{equation*}
   d(\mu)^{\alpha-1} \longrightarrow \infty \qquad \text{as } d(\mu)\downarrow0.
\end{equation*}

This proves \eqref{eq:derivative_singularity}. On the equilibrium side, $\mathcal H_\infty$ is identically zero under \eqref{eq:idealized_equilibrium_side}, so its derivative vanishes away from $\mu_c$.
\end{proof}

\begin{remark}
For the Hopf normal form of Case~A, the hypotheses \eqref{eq:idealized_birth_law}--\eqref{eq:birth_law_remainder} of Proposition~\ref{prop:consistency} hold in the idealized limit with remainder $r\equiv0$: by Proposition~\ref{prop:normal_form_lyapunov} below, $\mathcal H_\infty(\mu)=\kappa\sqrt{\mu}$ exactly on the periodic side (\eqref{eq:H_normal_form_scaling}), so \eqref{eq:idealized_birth_law} holds with $\alpha=1/2$ and no correction term, and \eqref{eq:birth_law_remainder} holds trivially. This is consistent with the near-unit $R^2$ obtained in the numerical birth-law fit of Figure~\ref{fig:prop1_birthlaw}.
\end{remark}

For a generic supercritical Hopf bifurcation, the amplitude of the attracting periodic orbit scales as $|\mu-\mu_c|^{1/2}$ near the critical parameter \cite{Kuznetsov2004}. In particular, for the Hopf normal form considered in Case~A, the limit-cycle radius satisfies the square-root scaling law exactly at leading order. Since the embedding parameters are held fixed and Euclidean rescaling multiplies all Vietoris--Rips filtration scales, and hence all persistence lifetimes, by the same factor, this provides the motivation for the idealized exponent $\alpha=1/2$ in Proposition~\ref{prop:consistency}.

This scaling prediction is examined numerically in Figure~\ref{fig:prop1_birthlaw}. On the periodic side, fitting $\mathcal H(\mu)=\kappa\mu^\alpha$ over $\mu\in(0.02,0.5)$ gives $\widehat{\alpha}=0.5000$ and $\widehat{\kappa}=1.7225$ (values rounded to four decimal places, consistent with the precision reported throughout this section), with $R^2\simeq1$. Thus, in the canonical Hopf normal form, the persistence functional recovers numerically the square-root onset predicted by the local normal-form geometry.

The proposition concerns the idealized continuous-parameter functional $\mathcal H_\infty$. It does not imply that the maximizer of the derivative of a smoothed finite-grid estimate is exactly equal to $\mu_c$. In the computational implementation, parameter discretization, finite time-series sampling, residual transients, delay reconstruction, persistence estimation, and smoothing can all affect the location of the numerical maximum of $|\widehat{\mathcal H}'(\mu)|$.

\begin{figure}[!ht]
\centering
\includegraphics[width=0.8\textwidth]{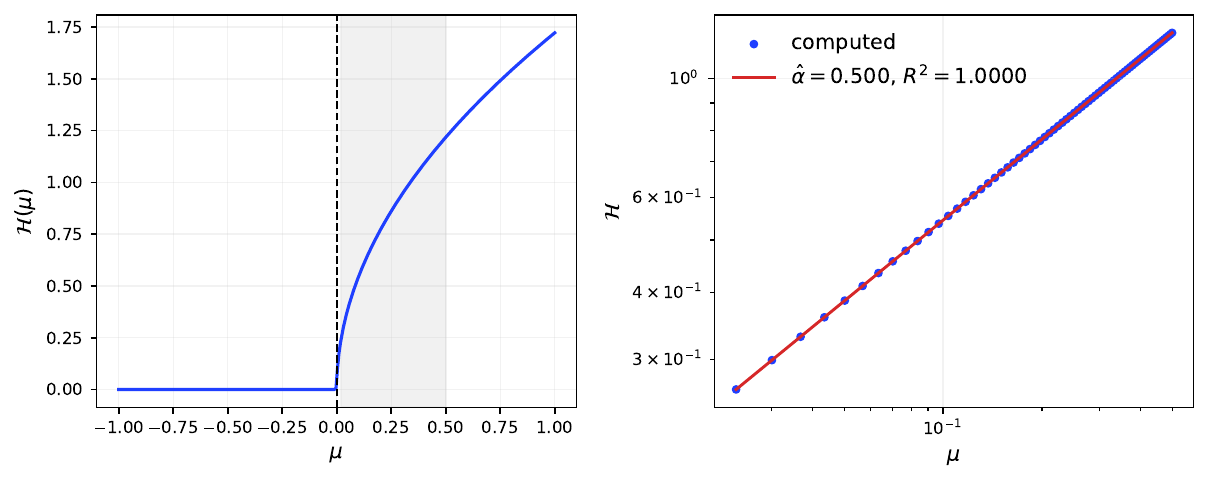}
\caption{
Empirical verification of the one-sided birth law in Proposition~\ref{prop:consistency} for the Hopf normal form. On the periodic side, the dominant persistence functional follows closely a power law $\mathcal H(\mu)=\kappa\mu^\alpha$. A log--log fit over $\mu\in(0.02,0.5)$ yields $\widehat{\alpha}\simeq0.5000$, in agreement with the square-root scaling predicted by the normal-form geometry. The figure supports the  population-level birth-law assumption and should not be interpreted as a finite-grid consistency result for the critical-parameter estimator.
}
\label{fig:prop1_birthlaw}
\end{figure}

Proposition~\ref{prop:consistency} establishes localization only at the idealized continuous-parameter level: under the assumed one-sided birth law, the singular onset of $\left|\mathcal H_\infty'(\mu)\right|$ occurs at the same critical parameter $\mu_c$ that separates the equilibrium and periodic regimes in Theorem~\ref{thm:hopf}. This population-level statement should not be interpreted as an exact finite-grid or finite-sample localization theorem.

For the Hopf normal form, the final numerical experiment, whose computational setup is detailed in Section~\ref{sec:methodology} and Section~\ref{sec:experiments}, uses $M=300$ uniformly spaced parameter values on $[-1,1]$ and yields 

\begin{equation}
    \widehat{\mu}^{\,\mathrm{TDA}}_c =0.016722,
    \qquad
    \left|\widehat{\mu}^{\,\mathrm{TDA}}_c-\mu_c\right| =0.016722.
    \label{eq:hopf_final_estimate}
\end{equation}

A useful control calculation isolates the role of finite parameter resolution and smoothing. Applying exactly the same parameter grid, Savitzky--Golay rule, numerical differentiation, and boundary exclusion to the idealized square-root onset $\mathcal H_{\mathrm{ideal}}(\mu) =\widehat{\kappa}\sqrt{\max(\mu,0)}$ produces the same numerical maximizer, $\widehat{\mu}^{\,\mathrm{TDA}}_c=0.016722$. Thus, in this canonical experiment, the observed positive displacement of the derivative maximum is reproduced by the combined action of the finite parameter grid and the regularization of the square-root singularity. This observation does not constitute a general finite-sample error formula: finite observation time, residual transients, delay reconstruction, sampling density, and persistent homology estimation can contribute additional discrepancies in more general systems. The dependence on parameter-grid resolution, temporal sampling, and smoothing bandwidth is quantified numerically in the finite-resolution study of Case~A.

The one-sided birth law is an explicit assumption of Proposition~\ref{prop:consistency}, and its conclusion should not be transferred automatically to systems for which the corresponding local scaling has not been established. The Lorenz experiment, in particular, lies outside the attracting periodic orbit assumptions of Theorem~\ref{thm:hopf}, because the classical Hopf bifurcation of its nonzero equilibria is subcritical. Its numerical localization is therefore interpreted as empirical detection of a strong topological reorganization rather than as a direct application of the theorem.

The BZ system provides a complementary orientation of the supercritical scenario: in the parameter convention used there, the attracting oscillatory branch lies on the $b<b_c$ side. This orientation is
compatible with the sign parameter $s$ introduced in Proposition~\ref{prop:consistency}. Nevertheless, no general square-root law for the reconstructed topological functional is assumed for the BZ experiment, and its critical-parameter localization is reported directly from the numerical estimator.

The estimator $\widehat{\mu}^{\,\mathrm{TDA}}_{c}$ therefore identifies the sampled parameter value at which the smoothed dominant topological functional changes most rapidly. When the hypotheses of Theorem~\ref{thm:hopf} and the local birth law of Proposition~\ref{prop:consistency} are satisfied, this provides a principled mechanism for localizing the Hopf transition within the finite resolution of the numerical procedure. Outside this regime, $\widehat{\mu}^{\,\mathrm{TDA}}_{c}$ should instead be interpreted as the location of the strongest observed change in the dominant one-dimensional persistence functional, rather than as a certification that the detected transition is necessarily a Hopf bifurcation. The procedure remains data-driven: it uses the observed time series, fixed delay coordinate reconstructions, and the corresponding persistence diagrams, without requiring the  governing equations or a spectral analysis. Robustness to finite-data perturbations is quantified by Theorem~\ref{thm:hopf}, Corollary~\ref{cor:gaussian_noise}, and the finite-resolution separation criterion of Corollary~\ref{cor:finite_resolution}.

\section{Computational Implementation}
\label{sec:methodology}

The computational implementation follows the same general pipeline for all systems considered in this work. For each value of the control parameter, a scalar time series is generated or observed, a fixed
delay-coordinate reconstruction is constructed, one-dimensional Vietoris--Rips persistent homology is computed, and the dominant topological functional $\mathcal H$ is evaluated. The resulting parameter-dependent functional is subsequently smoothed and differentiated according to Definition~\ref{def:estimator} to obtain the numerical estimate of the transition parameter. All dynamical systems were integrated using a fourth-order Runge--Kutta scheme. Delay-coordinate reconstructions were implemented explicitly from the scalar observations, and persistent homology was computed using \texttt{ripser}. All Vietoris--Rips computations use Euclidean distances, consistently with the geometric setting of Theorem~\ref{thm:hopf}. The same numerical pipeline was used throughout the final experiments unless a numerical parameter was intentionally varied as part of a sensitivity analysis.

\subsection{Time-series generation and delay reconstruction}
\label{subsec:timeseries}

For each value $\theta_j$ of the control parameter, where $\theta=\mu$, $\rho$, or $b$ depending on the system, the corresponding ordinary differential equation was integrated using the classical fourth-order Runge--Kutta method on a uniform temporal grid. A scalar observable $x(t;\theta_j)$ was  retained after removal of the prescribed initial transient.

The phase-space reconstruction was formed explicitly as

\begin{equation*}
 X_{\theta_j} = \left\{\bigl(x_k, x_{k-\tau},\ldots, x_{k-(m-1)\tau}\bigr)\right\},
\end{equation*}

which discretizes the delay embedding \eqref{eq:embedding} and instantiates the point cloud construction \eqref{eq:cloud} on the sampled time series, where $\tau$ is measured in samples. When indicated, a fixed row stride was subsequently applied to reduce the number of embedded points.

For each dynamical system, the embedding parameters $(\tau,m)$ were selected once from representative time series using mutual information \cite{Fraser1986} and False Nearest Neighbors \cite{kennel1992} as initial diagnostics, and were then held fixed throughout the complete parameter sweep. This avoids introducing parameter-dependent changes in the reconstruction map itself. The parameter grids were generated directly as uniformly spaced numerical arrays and were not rounded prior to simulation, reconstruction, or persistent-homology computation. The final baseline numerical configurations are summarized in Table~\ref{tab:computational_settings}.

\begin{table}[!ht]
\centering
\caption{Baseline numerical settings used in the final parameter sweeps. $T_{\mathrm{keep}}$ denotes the physical duration of the retained post-transient segment.}
\label{tab:computational_settings}
\scriptsize
\begin{tabular}{lcccccccc}
\toprule
System & Parameter range & $M$ & $T_{\mathrm{total}}$ & $f_s$ & $T_{\mathrm{keep}}$ & $\tau$ & $m$ & stride \\
\midrule
Hopf & $[-1,1]$ & 300 & 800 & 100 & 20 & 26 & 2 & 2 \\
Lorenz & $[20,30]$ & 300 & 100 & 130 & 15.38 & 16 & 2 & 2 \\
BZ & $[2,5]$ & 300 & 6000 & 100 & 40 & 57 & 2 & 5 \\
\bottomrule
\end{tabular}
\end{table}

The initial conditions were $(0,1.01)$ for the Hopf normal form, $(6.74,6.74,22.74)$ for the Lorenz system, and $(2,5.01)$ for the BZ model. The comparatively long integration horizons used for the Hopf and BZ systems were introduced to control slow transient relaxation close to their respective critical parameters; the corresponding diagnostics are discussed with the numerical experiments in Cases~A and~C.

\paragraph{Gaussian-noise experiments.}
When noise robustness was evaluated, the observed scalar series was generated as

\begin{equation*}
   \widetilde{x}_k = x_k+\nu s_x Z_k, \qquad Z_k\overset{\mathrm{iid}}{\sim}\mathcal N(0,1),
\end{equation*}

where $s_x$ denotes the empirical standard deviation of the corresponding noise-free time series and $\nu\geq0$ is the relative noise level. We use $\nu$ rather than $\sigma$ for the noise amplitude in order to distinguish it from the parameter $\sigma$ of the Lorenz system. Thus, in the notation of Corollary~\ref{cor:gaussian_noise}, the standard deviation of the observational perturbation is $s_\mu=\nu s_x$. 

Gaussian perturbations are not uniformly bounded almost surely, whereas hypothesis~\textnormal{(H3)} of Theorem~\ref{thm:hopf} is deterministic. Corollary~\ref{cor:gaussian_noise} connects the finite Gaussian-noise experiments with the deterministic stability result through a high-probability bound on the maximum observational perturbation.

For reproducibility, deterministic but independent random seeds were used across distinct parameter values. At a fixed parameter value, the same standard-normal realization was rescaled when comparing different positive values of $\nu$, yielding paired comparisons across noise amplitudes. One realization was used for each parameter and noise-level configuration. Consequently, the numerical experiments quantify stability across the tested noise levels and realizations rather than constituting a statistical characterization over an ensemble of noise realizations.

\subsection{Numerical estimator}
\label{subsec:numerical_estimator}

For each reconstructed point cloud $X_{\theta_j}$, one-dimensional Vietoris--Rips persistent homology was computed using \texttt{ripser} with the Euclidean metric. Only finite one-dimensional persistence intervals, $D_{\theta_j}^{(1)} = \{(b_{i,j},d_{i,j})\}_{i\in I_j}$, were retained for the numerical summaries. 

The principal observable is the numerical instantiation of the dominant topological functional of Definition~\ref{def:dominant_functional},

\begin{equation}
 \mathcal H(\theta_j) = \max_{i\in I_j}(d_{i,j}-b_{i,j}),
 \label{eq:numerical_H}
\end{equation}

with the convention $\mathcal H(\theta_j)=0$ when no finite off-diagonal $H_1$ class is present.

As a complementary aggregate summary, we also consider the $L^1$ norm of the one-dimensional Betti curve,

\begin{equation}
 \|\beta_1\|_{L^1}(\theta_j)=\int_0^\infty \beta_1(\varepsilon;\theta_j)\,d\varepsilon=\sum_{i\in I_j}(d_{i,j}-b_{i,j}).
 \label{eq:betti_l1}
\end{equation}

Thus, $\mathcal H$ records only the persistence of the dominant one-dimensional class, whereas $\|\beta_1\|_{L^1}$ aggregates the lifetimes of all finite $H_1$ classes. The latter is used only as a complementary topological descriptor and does not enter the critical-parameter estimator. For the estimator of Definition~\ref{def:estimator}, the sampled sequence $\mathcal H(\theta_j)$ was smoothed using a Savitzky--Golay filter of polynomial order $3$. The baseline window length was selected as the nearest valid odd integer to $5\%$ of the number of parameter values $M$, with a minimum window of five points. For the baseline experiments with $M=300$, this rule gives a window length of $15$ points.

The derivative of the smoothed functional was evaluated numerically using the actual parameter coordinates, $\widehat{\mathcal H}'(\theta_j) \approx \frac{d\widehat{\mathcal H}}{d\theta}(\theta_j)$. To reduce endpoint effects, half of the Savitzky--Golay window was excluded from each side of the parameter interval before searching for the derivative maximum. Thus, for the baseline window of $15$ points, seven parameter values were excluded at each boundary. The estimate was then defined by

\begin{equation}
 \widehat{\theta}^{\,\mathrm{TDA}}_c = \theta_{j^\ast}, \qquad j^\ast = \underset{j\in\mathcal J_{\mathrm{int}}}{\arg\max} \left| \widehat{\mathcal H}'(\theta_j) \right|.
 \label{eq:numerical_estimator}
\end{equation}

The same smoothing and boundary rule was used for all three baseline systems and for the finite-resolution experiments, except when the smoothing bandwidth itself was intentionally varied as part of the sensitivity analysis.

Throughout Section~\ref{sec:experiments}, numerical values of $\widehat\theta^{\,\mathrm{TDA}}_c$ are reported to six decimal places for reproducibility of the computational pipeline. This displayed precision is substantially finer than the true localization accuracy, which is governed instead by the finite-resolution bounds of Theorem~\ref{thm:hopf} and Corollary~\ref{cor:finite_resolution} and by the parameter-grid spacings reported in Table~\ref{tab:computational_settings} and Table~\ref{tab:resolution_sensitivity}; the number of reported digits should not be read as a claim about the accuracy of the localization itself.

\subsection{Implementation Algorithm}

The Algorithm \ref{algorithm} summarizes the computational implementation of the proposed topological criterion and provides a practical procedure for detecting Hopf bifurcations from time series data:

\begin{algorithm}[H]
\caption{Topological Detection of Hopf Bifurcations: Derivative-Based Topological Estimator}
\label{algorithm}
\begin{algorithmic}

\STATE \textbf{Input:}
ordered parameter grid
$\{\mu_j\}_{j=1}^{M}$;
family of observed time series
$\{x_{\mu_j}(t)\}_{j=1}^{M}$;
fixed embedding parameters $(\tau,m)$ and row stride;
smoothing parameters.

\STATE \textbf{Output:}
estimated topological transition parameter
$\widehat{\mu}^{\,\mathrm{TDA}}_c$.

\STATE \texttt{/* Topological feature extraction */}

\FOR{$j=1,\ldots,M$}

   \STATE $X_{\mu_j} \leftarrow$ delay-coordinate embedding of $x_{\mu_j}(t)$ using the same $(\tau,m)$ and row stride throughout the parameter sweep

    \STATE
    $D^{(1)}_{\mu_j}
    \leftarrow$
    one-dimensional Euclidean Vietoris--Rips persistence diagram of $X_{\mu_j}$

    \IF{$D^{(1)}_{\mu_j}$ contains a finite off-diagonal point}
        \STATE
        $\mathcal H(\mu_j)
        \leftarrow
        \displaystyle
        \max_{(b,d)\in D^{(1)}_{\mu_j}}(d-b)$
    \ELSE
        \STATE
        $\mathcal H(\mu_j)\leftarrow0$
    \ENDIF

\ENDFOR

\STATE \texttt{/* Smoothing and differentiation */}

\STATE
$\widehat{\mathcal H}
\leftarrow$
Savitzky--Golay smoothing of
$\{(\mu_j,\mathcal H(\mu_j))\}_{j=1}^{M}$

\STATE $\widehat{\mathcal H}'\leftarrow$ numerical derivative of $\widehat{\mathcal H}$ using the actual parameter grid $\{\mu_j\}_{j=1}^{M}$
\STATE $w\leftarrow$ Savitzky--Golay window length
\STATE $r\leftarrow\lfloor w/2\rfloor$
\STATE $\mathcal J_{\mathrm{int}} \leftarrow \{1+r,\ldots,M-r\}$
\STATE \texttt{/* Critical-transition localization */}

\STATE $j^\ast \leftarrow \displaystyle \arg\max_{j\in\mathcal J_{\mathrm{int}}} \left| \widehat{\mathcal H}'(\mu_j) \right|$

\STATE
$\widehat{\mu}^{\,\mathrm{TDA}}_c \leftarrow \mu_{j^\ast}$

\STATE \textbf{return}
$\widehat{\mu}^{\,\mathrm{TDA}}_c$

\end{algorithmic}
\end{algorithm}

Algorithm~\ref{algorithm} is the numerical implementation of Definition~\ref{def:estimator}, computing the estimator \eqref{eq:numerical_estimator} from the discretized functional \eqref{eq:numerical_H}. It should be distinguished from the finite-resolution threshold statement of Corollary~\ref{cor:finite_resolution}. Theorem~\ref{thm:hopf} and Corollary~\ref{cor:finite_resolution} provide geometric detectability bounds for the supercritical attracting-cycle regime, whereas Algorithm~\ref{algorithm} provides a practical localization rule based on the point of steepest change of the observed topological functional. Under the hypotheses of Proposition~\ref{prop:consistency}, these two descriptions are centered at the same critical parameter at the idealized population level. Outside that regime, the output of Algorithm~\ref{algorithm} should be interpreted as the location of the strongest observed topological transition.

\paragraph{Computational Complexity.}
Let $N$ denote the number of points in each reconstructed time series and $M$ the number of parameter values. For fixed embedding dimension, the delay-coordinate reconstruction has linear cost in $N$. Persistent homology is the computationally dominant step. Even when only $H_1$ is required, a complete Vietoris--Rips $2$-skeleton may contain $O(N^2)$ edges and $O(N^3)$ triangles, and the cost of persistence reduction depends strongly on the geometry of the data and on implementation details.

Modern implementations such as \texttt{ripser} use cohomology-based computations, clearing, and related optimizations that substantially reduce the computational burden in practice \cite{Bauer2021}. We therefore do not assign a universal $O(N^2)$ complexity to the persistent-homology step. Since persistence is computed independently for each of the $M$ parameter values, the overall computational cost is $M$ times the cost of a single persistence computation, in addition to the comparatively inexpensive reconstruction, smoothing, and differentiation steps.

The code, notebooks, processed numerical outputs, and figure-generation scripts used in this work are publicly available at \url{https://github.com/JhonathanBarrios21/hopf-tda}.

\section{Experiments and Results}
\label{sec:experiments}

In this section, we examine the proposed topological functional and derivative-based localization procedure on three dynamical systems with distinct roles. Case~A considers the supercritical Hopf normal form, which provides the canonical setting corresponding to the theoretical mechanism developed in Section~\ref{sec:criterion}. Case~B considers the classical Lorenz system as an empirical stress test outside the attracting-periodic-orbit assumptions of the main theorem. Case~C considers a reduced BZ model with an attracting oscillatory branch on the opposite side of the chosen parameter orientation. We additionally examine finite-resolution sensitivity and compare the topological functional descriptively with Lyapunov and Floquet quantities. Each case follows the time-series generation and reconstruction procedure of Section~\ref{subsec:timeseries} and the smoothing and localization procedure of Section~\ref{subsec:numerical_estimator}, with the system-specific parameters summarized in Table~\ref{tab:computational_settings}.

\subsection*{Case A: Normal form of Hopf bifurcation}

As a first experiment for the validation of the proposed topological criterion, the normal form of the supercritical Hopf bifurcation was considered, given by the system:

\[
\dot{x} = \mu x - \omega y - x(x^2 + y^2),
\qquad
\dot{y}=\omega x + \mu y - y(x^2+y^2);
\]

where $\mu$ is the control parameter and $\omega > 0$ represents the angular frequency of the system. This system exhibits a Hopf bifurcation at the critical value $\mu_c = 0$. For $\mu < 0$, the origin is a stable equilibrium, while for $\mu > 0$, a stable limit cycle emerges. This example allows us to evaluate, in a controlled environment, whether the criterion based on persistent homology is capable of correctly detecting the transition between a stationary regime and an oscillatory regime from time series.

The parameter interval $\mu\in[-1,1]$ was sampled at $M=300$ uniformly spaced values. For each value of $\mu$, the system was integrated with $\omega=6$, sampling frequency $f_s=100$, and total integration time $T_{\mathrm{total}}=800$. The last $T_{\mathrm{keep}}=20$ time units were retained for the analysis. The scalar observable $x(t)$ was reconstructed using fixed delay parameters $\tau=26$ samples and $m=2$, with row stride $2$. One-dimensional Euclidean Vietoris--Rips persistent homology was then computed and the dominant functional $\mathcal H(\mu)$ was evaluated.

The comparatively long integration horizon was used to control the critical slowing down of trajectories on the equilibrium side of the bifurcation. In polar coordinates, the radial dynamics are $\dot r=r(\mu-r^2)$,  so for $\mu<0$ and $|\mu|\ll1$ the asymptotic relaxation rate toward the origin is of order $|\mu|$. Consequently, the corresponding relaxation time increases as the critical parameter is approached. A transient diagnostic showed that the residual precritical persistence decreased systematically when the integration horizon was increased from $T_{\mathrm{total}}=100$ to $800$. We therefore adopted $T_{\mathrm{total}}=800$ for the final Hopf calculations in order to reduce finite-transient contamination near $\mu_c$, rather than tuning the trajectory length to the critical-parameter estimate.

\subsection*{Results}

Figure~\ref{fig:hopf_pd} illustrates the change in reconstructed topology across representative values of the control parameter. Away from the critical region on the equilibrium side, no persistent one-dimensional structure is present. As the bifurcation is approached and the attracting periodic orbit develops for $\mu>0$, a dominant finite $H_1$ class emerges and its persistence increases with the size of the reconstructed cycle. The diagrams provide a geometric illustration of this transition, while quantitative comparisons across parameter values are made through the scalar functional $\mathcal H(\mu)$.

\begin{figure}[!ht]
\centering
\includegraphics[width=0.8\textwidth]{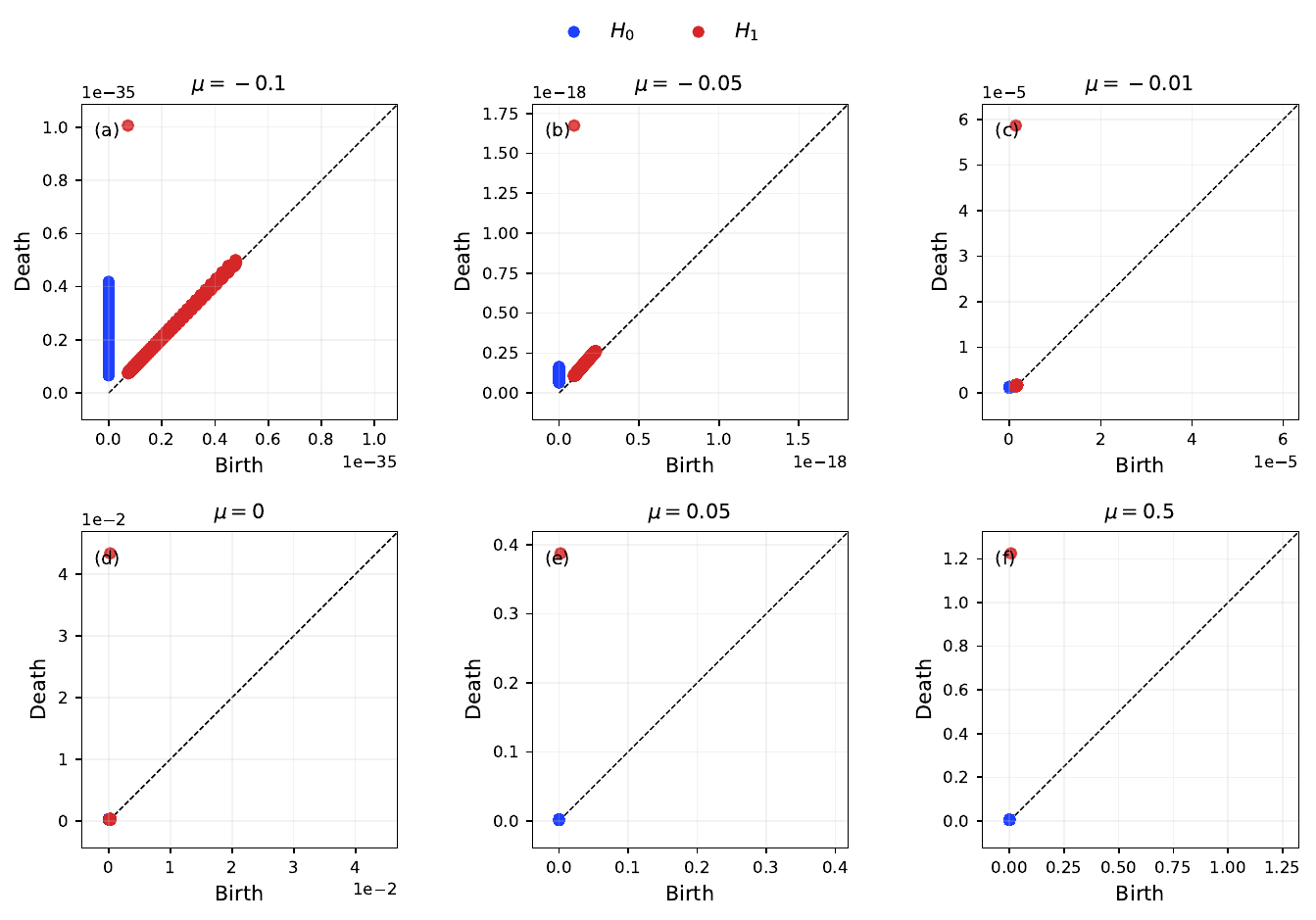}
\caption{Representative persistence diagrams for the Hopf normal form at $\mu=-0.10,-0.05,-0.01,0,0.05$, and $0.50$. The equilibrium side contains no robust persistent $H_1$ class, whereas a dominant one-dimensional feature develops on the periodic side as the attracting limit cycle grows. Axes are independently scaled in each panel because the characteristic persistence scales vary by several orders of magnitude across the bifurcation. In particular, for $\mu$ well below $\mu_c$ the plotted persistence values are of order $10^{-30}$--$10^{-2}$; at this scale the points reflect floating-point roundoff in the numerical pipeline rather than genuine topological structure, consistent with the theoretical prediction $\mathscr H(P_\mu)=0$ for $\mu<\mu_c$ in Theorem~\ref{thm:hopf}. Distances from the diagonal should therefore be interpreted within each panel rather than compared visually across panels; quantitative comparisons across parameter values are performed using $\mathcal H(\mu)$.
}
\label{fig:hopf_pd}
\end{figure}

The same transition is quantified by the topological observables in Figure~\ref{fig:hopf_functionals}. The primary quantity, $\mathcal H(\mu)=\max(d-b)$, measures the persistence of the dominant one-dimensional class. On the equilibrium side it remains negligible away from the immediate critical neighborhood, whereas for $\mu>0$ it increases according to the growth of the reconstructed periodic orbit. As a complementary descriptor, we consider the integrated one-dimensional Betti curve $\|\beta_1\|_{L^1}$, defined in \eqref{eq:betti_l1}. Unlike $\mathcal H$, this quantity aggregates all finite $H_1$ lifetimes. Its evolution is therefore complementary rather than an independent statistical validation of the dominant-persistence functional. In the Hopf normal form, a single dominant $H_1$ class accounts for most of the total persistence on the periodic side, so the two curves in Figure~\ref{fig:hopf_functionals} closely track one another; this near-coincidence reflects the comparative topological simplicity of this system and need not hold for the richer reconstructions considered in Cases~B and~C below.

\begin{figure}[!ht]
\centering
\includegraphics[width=0.65\textwidth]{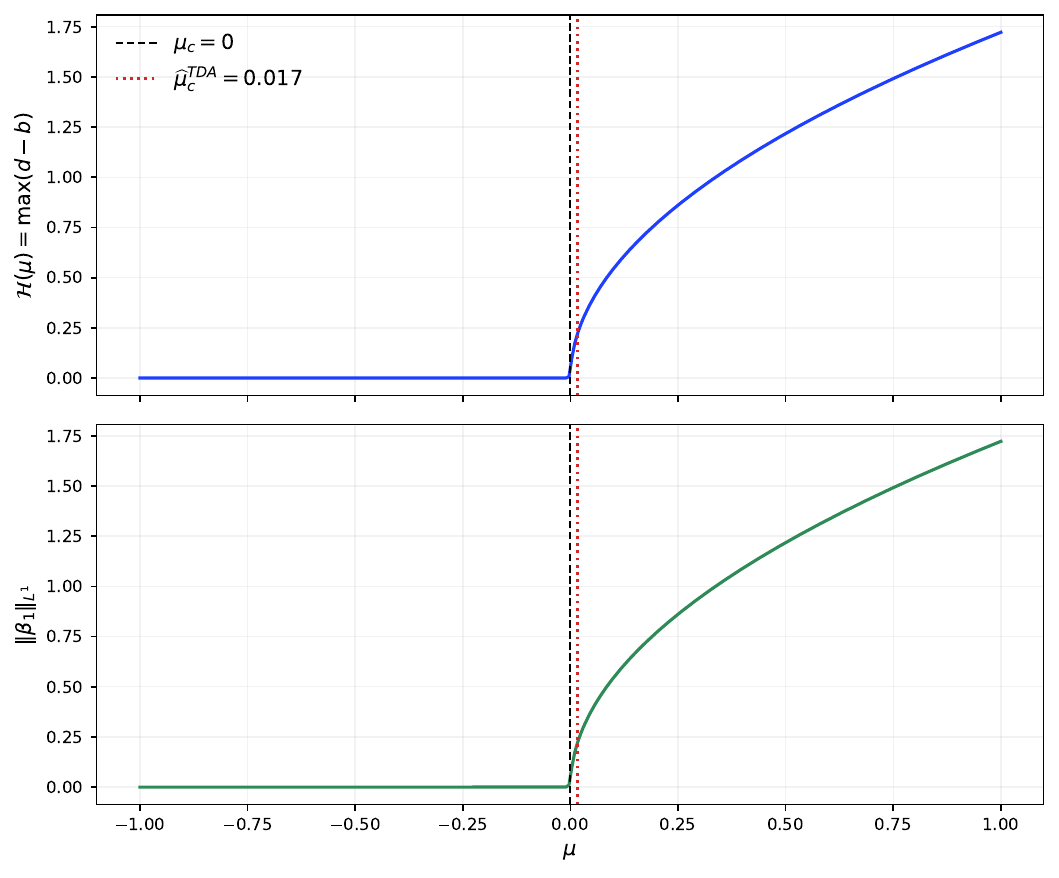}
\caption{ Topological observables for the Hopf normal form. The dominant persistence functional $\mathcal H(\mu)$ records the lifetime of the most persistent finite $H_1$ class, whereas $\|\beta_1\|_{L^1}$ is the integrated one-dimensional Betti curve, equivalently the sum of all finite $H_1$ persistence lifetimes. Both quantities reflect the transition from the equilibrium regime to the attracting periodic regime, although they summarize different aspects of the persistence diagram. }
\label{fig:hopf_functionals}
\end{figure}

Applying the topological estimator of Definition~\ref{def:estimator} to the baseline parameter grid yields $\widehat{\mu}^{\,\mathrm{TDA}}_c=0.016722$,  compared with the analytical critical value $\mu_c=0$. The absolute localization error is therefore $0.016722$. This value should be interpreted as a finite-resolution localization of the transition rather than as an exact finite-grid recovery of the theoretical Hopf point.
Figure~\ref{fig:hopf_estimator} shows the smoothed dominant persistence functional and the magnitude of its numerical derivative used by the estimator. The positive displacement of the derivative maximum from $\mu_c$ is consistent with the finite-grid and smoothing effects discussed after Proposition~\ref{prop:consistency}. Its dependence on the numerical resolution is examined explicitly below.

\begin{figure}[!ht]
\centering
\includegraphics[width=0.65\textwidth] {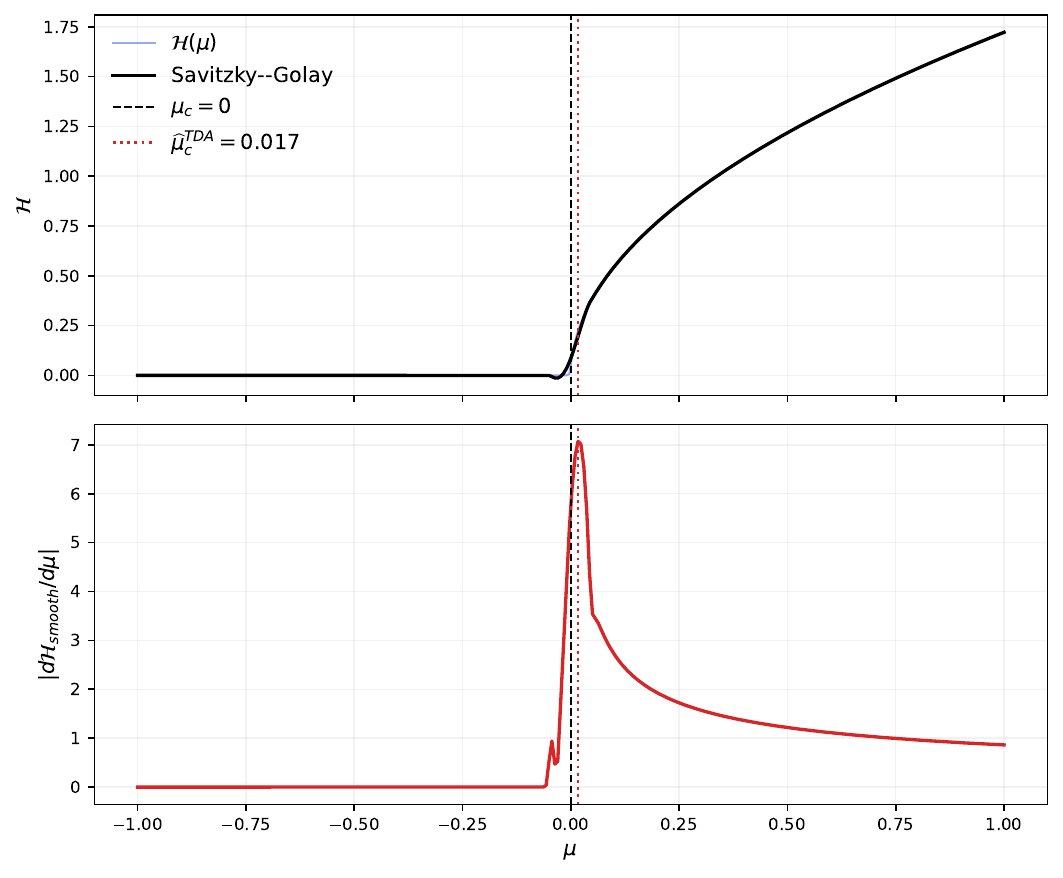}
\caption{Derivative-based localization of the Hopf transition. The dominant persistence functional is smoothed using the baseline Savitzky--Golay rule and differentiated with respect to the actual parameter grid. The reference value $\mu_c=0$ and the numerical estimate $\widehat{\mu}^{\,\mathrm{TDA}}_c=0.016722$ are indicated. The finite positive displacement of the derivative maximum is examined further in the finite-resolution sensitivity analysis.}
\label{fig:hopf_estimator}
\end{figure}

Because the Hopf normal form provides a controlled setting, we also examined sensitivity to the fixed reconstruction parameters. Figure~\ref{fig:embedding_sensitivity} shows that the absolute persistence scale depends on the selected delay and embedding dimension, as expected because these parameters change the geometry of the reconstructed point cloud. Nevertheless, the qualitative one-sided transition of $\mathcal H(\mu)$ is preserved across all tested fixed choices. This supports the use of a common reconstruction map throughout each parameter sweep and avoids introducing parameter-dependent changes in the embedding itself.

\begin{figure}[!ht]
\centering
\includegraphics[width=0.85\textwidth]
{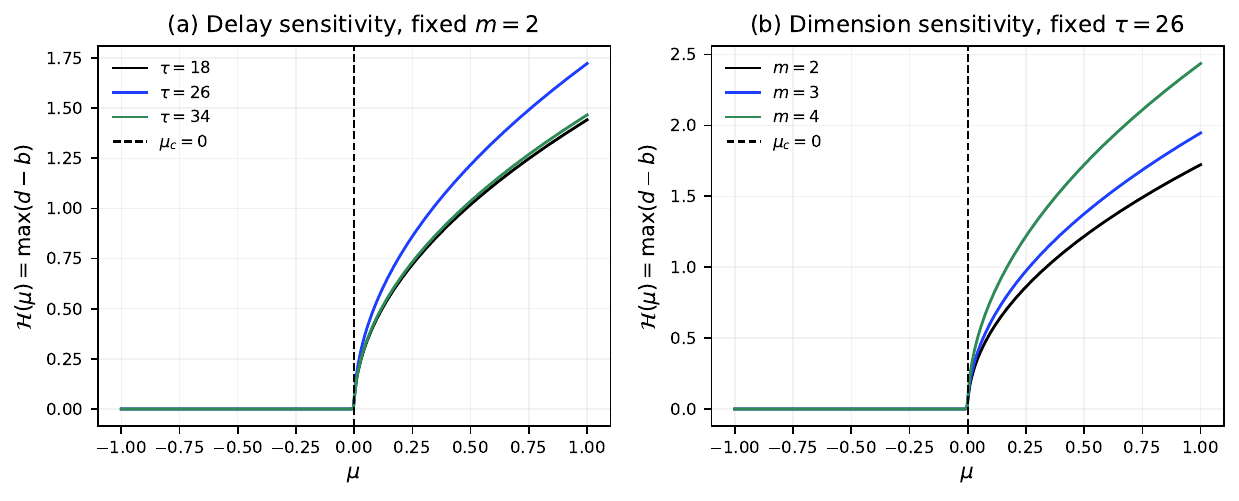}
\caption{ Sensitivity of the dominant topological functional $\mathcal H(\mu)$ to fixed choices of the delay-embedding parameters in the Hopf normal form. (a) Delay sensitivity for $\tau\in\{18,26,34\}$ with fixed embedding dimension $m=2$. (b) Dimension sensitivity for $m\in\{2,3,4\}$ with fixed delay $\tau=26$. Although the magnitude of $\mathcal H$ depends on the reconstruction parameters, the qualitative one-sided transition and its location near the reference value $\mu_c=0$ are preserved across the tested choices. Each pair $(\tau,m)$ is held fixed throughout its complete parameter sweep. }
\label{fig:embedding_sensitivity}
\end{figure}

\paragraph{Finite-resolution sensitivity.}
Because the Hopf normal form has the analytically known critical value $\mu_c=0$, it provides a controlled setting for examining the numerical resolution of the derivative-based estimator. We separately varied the parameter-grid resolution, the Savitzky--Golay smoothing bandwidth, and the temporal sampling frequency. These calculations are interpreted as a finite-resolution sensitivity analysis rather than as a formal finite-sample convergence study.
Table~\ref{tab:resolution_sensitivity} summarizes the resulting parameter-grid, smoothing, and temporal-sampling sensitivity analyses.

\begin{table}[!ht]
\centering
\caption{ Finite-resolution sensitivity of the Hopf critical-parameter estimator. Each panel varies one numerical component while keeping the remaining baseline choices fixed. In the temporal-sampling experiment, the physical delay was fixed at $\tau_{\mathrm{physical}}=0.26$ and the retained observation horizon at $20$ time units.}
\label{tab:resolution_sensitivity}
\scriptsize
\begin{tabular}{ccccc}
\multicolumn{5}{c}{\textbf{(a) Parameter-grid sensitivity}}\\
\toprule
$M$ & $\Delta\mu$ & SG window &
$\widehat{\mu}^{\,\mathrm{TDA}}_c$ &
absolute error\\
\midrule
100 & 0.020202 & 5  & 0.010101 & 0.010101\\
200 & 0.010050 & 11 & 0.025126 & 0.025126\\
300 & 0.006689 & 15 & 0.016722 & 0.016722\\
600 & 0.003339 & 31 & 0.018364 & 0.018364\\
\bottomrule
\end{tabular}
\vspace{0.35cm}

\begin{tabular}{cccc}
\multicolumn{4}{c}{\textbf{(b) Smoothing sensitivity, $M=300$}}\\
\toprule
SG fraction & SG window &
$\widehat{\mu}^{\,\mathrm{TDA}}_c$ &
absolute error\\
\midrule
0.025 & 7  & 0.010033 & 0.010033\\
0.050 & 15 & 0.016722 & 0.016722\\
0.075 & 23 & 0.030100 & 0.030100\\
0.100 & 31 & 0.043478 & 0.043478\\
\bottomrule
\end{tabular}

\vspace{0.35cm}

\begin{tabular}{cccccc}
\multicolumn{6}{c}{\textbf{(c) Temporal-sampling sensitivity}}\\
\toprule
$f_s$ & $\Delta t$ &
$\tau$ (samples) & embedded points &
$\widehat{\mu}^{\,\mathrm{TDA}}_c$ &
absolute error\\
\midrule
50  & 0.020000 & 13 & 494  & 0.016722 & 0.016722\\
100 & 0.010000 & 26 & 987  & 0.016722 & 0.016722\\
150 & 0.006667 & 39 & 1481 & 0.016722 & 0.016722\\
\bottomrule
\end{tabular}
\end{table}

The grid experiment does not show monotone convergence of the estimate toward zero as $M$ increases. This is not unexpected because the baseline Savitzky--Golay window is defined as approximately $5\%$ of the parameter grid, so grid refinement simultaneously changes the number of smoothing points while maintaining a comparable smoothing scale in parameter space. The result should therefore be interpreted as stability under parameter-grid refinement at approximately fixed relative smoothing, not as a convergence-rate experiment.

The smoothing study shows a clearer systematic effect. Increasing the Savitzky--Golay window shifts the derivative maximum progressively toward positive values, from approximately $0.0100$ at a $2.5\%$ window to approximately $0.0435$ at a $10\%$ window. Thus the smoothing bandwidth is a material finite-resolution parameter, although it is not assumed to be the only possible source of localization error in general data.

Finally, when the temporal sampling frequency is varied from $50$ to $150$ while keeping the physical delay $\tau_{\mathrm{physical}}=0.26$ and the retained physical observation horizon fixed, the estimated value remains $\widehat{\mu}^{\,\mathrm{TDA}}_c=0.016722$ in all three cases. Within the tested range, the estimator is therefore insensitive to this change in temporal discretization once the physical reconstruction scale is held fixed.

\subsection*{Case B: Lorenz System}
As a second experiment, we consider the Lorenz system as a more challenging test of the topological functional in a higher-dimensional nonlinear setting. The system is

\[
\dot{x}=\sigma(y-x),
\qquad
\dot{y}=x(\rho-z)-y,
\qquad
\dot{z}=xy-\beta z,
\]

where $\sigma$, $\rho$, and $\beta$ are positive parameters. We use the classical values $\sigma=10$, $\beta=\frac{8}{3}$ and treat $\rho$ as the control parameter \cite{Lorenz1963}. For these values, the nonzero equilibria lose stability through a Hopf bifurcation at

\begin{equation}
  \rho_c=\frac{\sigma(\sigma+\beta+3)} {\sigma-\beta-1} = 24.736842\ldots .
 \label{eq:lorenz_hopf}
\end{equation}

The Lorenz system exhibits substantially more complex reconstructed geometry than the Hopf normal form, making it useful for assessing the behavior of the proposed topological observable beyond the canonical attracting-cycle setting.

For the classical parameter values $\sigma=10$ and $\beta=8/3$, the Hopf bifurcation of the nonzero equilibria is subcritical \cite{Hassard1981}. Consequently, the small periodic branch associated locally with the Hopf point is unstable, and the Lorenz experiment does not satisfy the attracting-periodic-orbit hypothesis~\textnormal{(H1)} of Theorem~\ref{thm:hopf}. We retain this case deliberately as an empirical stress test of the topological functional in a more complex dynamical setting. Its numerical localization near $\rho_c$ should therefore be interpreted as empirical detection of a strong geometric reorganization near the loss of equilibrium stability, rather than as a direct validation of Theorem~\ref{thm:hopf}.

The parameter interval $\rho\in[20,30]$ was sampled at $M=300$ uniformly spaced values. For each parameter value, the system was integrated for $T_{\mathrm{total}}=100$ with sampling frequency $f_s=130$. The last $2000$ scalar observations, corresponding to a retained nominal duration 

\[T_{\mathrm{keep}}=\frac{2000}{130}=15.384615, \]

were used for the topological analysis. The initial condition was $(6.74,6.74,22.74)$. Delay-coordinate reconstruction used $\tau=16$ samples, embedding dimension $m=2$, and row stride $2$.
The dominant persistence functional was evaluated for relative Gaussian noise levels $\nu\in\{0,0.04,0.10\}$, using the noise model described in Section~\ref{sec:methodology}.

Unlike the Hopf normal form and BZ experiments, no explicit transient relaxation diagnostic of the type introduced below for the BZ reaction (Eqs.~\eqref{eq:bz_relaxation_rate}--\eqref{eq:bz_relaxation_time}) was performed for the Lorenz system, and the same integration horizon $T_{\mathrm{total}}=100$ was used across the full parameter sweep. Because the Lorenz experiment is retained precisely as an empirical stress test lying outside the hypotheses of Theorem~\ref{thm:hopf}, rather than as a quantitative validation of it, this choice was not tuned to a relaxation-time estimate near $\rho_c$; a dedicated near-critical transient analysis for the Lorenz system, analogous to the one carried out below for the BZ reaction, is left for future work.

\subsection*{Results}

Figure~\ref{fig:lorenz_H_noise} shows the dominant persistence functional across the parameter sweep. For parameter values below the reference Hopf point, $\mathcal H(\rho)$ remains comparatively small. A pronounced increase occurs in the neighborhood of $\rho_c=24.736842\ldots$ (\eqref{eq:lorenz_hopf}), followed by substantially larger and more irregular persistence values as the reconstructed dynamics become more geometrically complex.

The same qualitative reorganization is visible for all three tested noise amplitudes. Measurement noise changes the detailed magnitude of the persistence curves, but the strongest topological change remains localized close to the reference transition within the tested configurations. Because one reproducible realization was used for each parameter and noise-level configuration, this result is interpreted as stability across the tested noise levels and realizations rather than as a general statistical robustness statement.

\begin{figure}[!ht]
\centering
\includegraphics[width=0.75\textwidth] {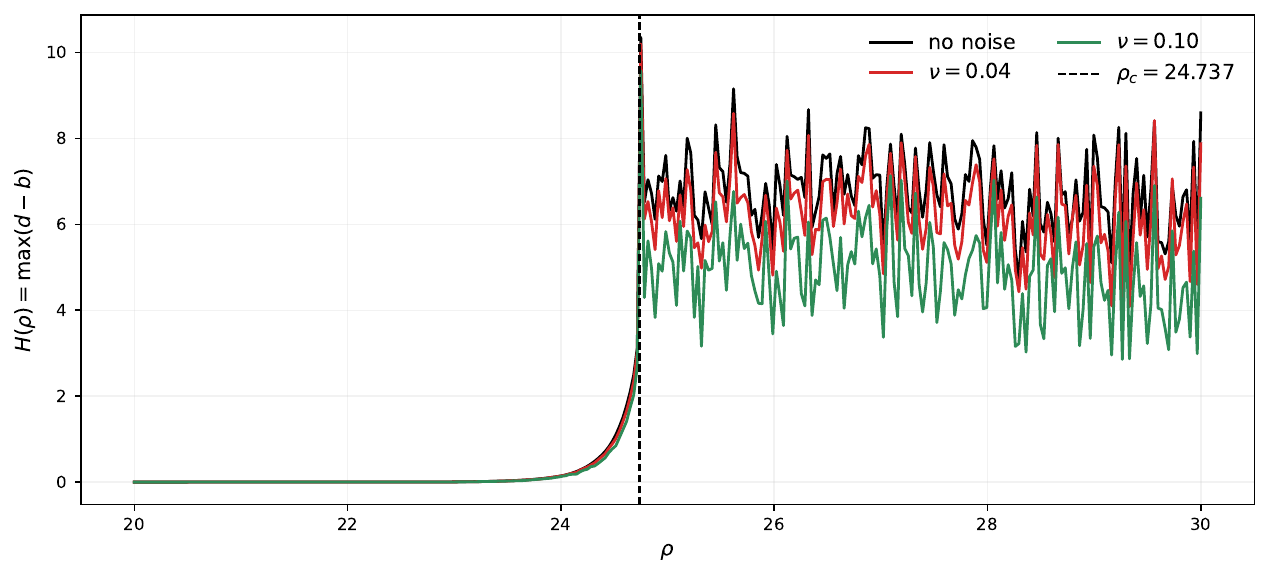}
\caption{ Dominant one-dimensional persistence functional $\mathcal H(\rho)=\max(d-b)$ for the Lorenz system under the tested relative Gaussian-noise levels $\nu\in\{0,0.04,0.10\}$. The vertical dashed line indicates the reference Hopf value $\rho_c=24.736842\ldots$. The curves exhibit a pronounced topological reorganization in the neighborhood of the loss of equilibrium stability. Since the classical Lorenz Hopf bifurcation is subcritical, this experiment is interpreted as an empirical stress test rather than as a direct application of Theorem~\ref{thm:hopf}. }
\label{fig:lorenz_H_noise}
\end{figure}

Applying Definition~\ref{def:estimator} gives

\begin{equation*}
\widehat{\rho}^{\,\mathrm{TDA}}_c =
\begin{cases}
24.648829, & \nu=0,\\
24.615385, & \nu=0.04,\\
24.615385, & \nu=0.10.
\end{cases}
\end{equation*}

Relative to the reference value $\rho_c=24.736842$, the corresponding absolute errors are $0.088013$, $0.121457$, and $0.121457$, respectively. The relative errors are approximately $0.356\%$, $0.491\%$, and $0.491\%$. Thus, all three tested configurations localize the strongest topological change within $0.5\%$ of the reference parameter.

Figure~\ref{fig:lorenz_estimator} makes explicit the operational localization step of the proposed procedure. The modest displacement of the estimated parameter under measurement noise indicates that the location of the strongest persistence change is relatively stable within the tested setup. This numerical behavior does not extend the theoretical guarantee of Theorem~\ref{thm:hopf} to the Lorenz system.

\begin{figure}[!ht]
\centering
\includegraphics[width=0.65\textwidth]{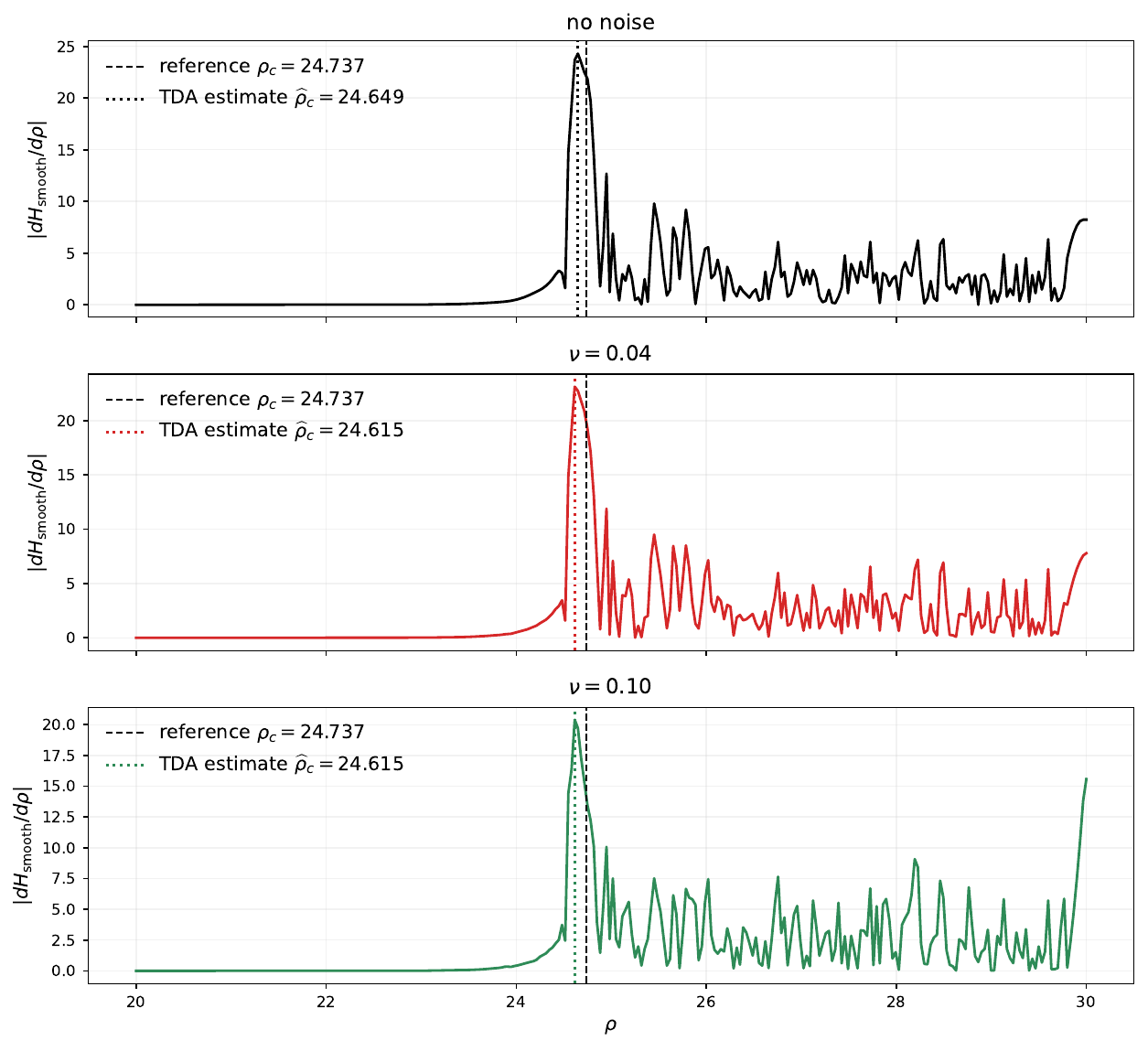}
\caption{ Derivative-based localization of the dominant topological change in the Lorenz system for the tested noise levels. The estimator selects the interior parameter value maximizing $|\widehat{\mathcal H}'(\rho)|$ after application of the common Savitzky--Golay smoothing and boundary-exclusion rules. The resulting estimates remain close to the reference value $\rho_c=24.736842\ldots$ across the tested configurations. }
\label{fig:lorenz_estimator}
\end{figure}

\subsection*{Case C: Belousov-Zhabotinsky Reaction}
As a final experiment to further assess the generality of the proposed topological criterion, we consider the Belousov--Zhabotinsky (BZ) reaction, which provides an example of oscillatory chemical dynamics exhibiting a transition that differs qualitatively from the cases analyzed previously. The BZ reaction is considerably complex, involving more than twenty elementary steps. However, many of these steps quickly reach equilibrium, allowing the complete kinetics to be reduced to a system. Its simplified mathematical model is:

\[
\dot{x} = a - x - \frac{4xy}{1 + x^2},
\qquad
\dot{y} = bx \left( 1 - \frac{y}{1 + x^2} \right);
\]

where $x$ and $y$ are dimensionless concentrations and $a,b>0$ are model parameters. Throughout this experiment we fix $a=10$ and use $b$ as the control parameter.

\begin{remark}[Linear stability and the reference value $b_c$]
\label{rem:bz_stability}

For $a=10$, the system has an equilibrium in the positive quadrant at $(x^*,y^*)=(2,5)$. The Jacobian evaluated at this equilibrium has

\begin{equation*}
    \operatorname{tr}J(b)
    =
    \frac{7}{5}-\frac{2}{5}b,
    \qquad
    \det J(b)=2b.
\end{equation*}

Hence the linear stability crossing occurs at

\begin{equation*}
    b_c=\frac{7}{2}=3.5,
\end{equation*}

where

\begin{equation*}
    \operatorname{tr}J(b_c)=0,
    \qquad
    \det J(b_c)=7>0.
\end{equation*}

The eigenvalues are therefore purely imaginary at $b_c$, and the crossing is transversal because

\begin{equation*}
    \frac{d}{db}\operatorname{tr}J(b)=-\frac{2}{5}\neq0.
\end{equation*}

We use $b_c=3.5$ as the reference Hopf value in the numerical analysis. The simulations exhibit an attracting periodic branch for $b<b_c$, which contracts toward the equilibrium as $b$ approaches $b_c$ from below. Thus, after reversing the parameter orientation, the observed branch is consistent with the supercritical attracting-cycle scenario considered in Theorem~\ref{thm:hopf}.
\end{remark}

In the parameter orientation used here (Remark~\ref{rem:bz_stability}), the oscillatory regime occurs for $b<b_c$, while the equilibrium is stable for $b>b_c$. Thus the dominant one-dimensional topological feature disappears, rather than appears, as $b$ is increased through the transition. Since the estimator uses the magnitude $|\widehat{\mathcal H}'(b)|$, its operational definition is invariant under this reversal of parameter orientation.

The interval $b\in[2,5]$ was sampled at $M=300$ uniformly spaced values. For each value of $b$, the system was integrated with sampling frequency $f_s=100$ and total integration time $T_{\mathrm{total}}=6000$. The last $T_{\mathrm{keep}}=40$ time units were retained for the topological analysis. The initial condition was $(2,5.01)$. Delay-coordinate reconstruction used $\tau=57$ samples, embedding dimension $m=2$, and row stride $5$. The dominant persistence functional was evaluated for relative Gaussian-noise levels $\nu\in\{0,0.02,0.04,0.06\}$,  using the noise model of Section~\ref{sec:methodology}.

The long integration horizon is essential in this experiment. The initial condition $(2,5.01)$ lies very close to the equilibrium $(2,5)$, and for $b<b_c$ close to the Hopf point the departure from this unstable equilibrium is extremely slow. The real part of the linearized eigenvalue pair is

\begin{equation}
    \operatorname{Re}\lambda(b)=\frac{1}{2}\operatorname{tr}J(b) = 0.7-0.2b = 0.2(3.5-b).
    \label{eq:bz_relaxation_rate}
\end{equation}

Thus the characteristic near-critical time scale behaves as

\begin{equation}
    \tau_{\mathrm{relax}} \sim \frac{1}{|\operatorname{Re}\lambda(b)|} = \frac{5}{|3.5-b|}.
    \label{eq:bz_relaxation_time}
\end{equation}

For the sampled value immediately below the critical point, $b\simeq3.49498$, this gives $\tau_{\mathrm{relax}}\simeq997$. A horizon of approximately six such time scales is therefore about $5980$, motivating the choice $T_{\mathrm{total}}=6000$ used in the final computation.

This transient control is important because shorter simulations can mistake slow departure from the near-critical equilibrium for the asymptotic dynamics. In preliminary calculations with $T_{\mathrm{total}}=100$, this effect produced a substantially displaced topological transition. The final integration horizon was therefore selected from the dynamical relaxation scale rather than by tuning the simulation to reproduce the known value $b_c=3.5$.

\subsection*{Results}
Figure~\ref{fig:bz_H_noise} shows the dominant persistence functional across the final BZ parameter sweep. For $b<b_c$, the reconstructed oscillatory regime produces a pronounced finite $H_1$ class and therefore positive values of $\mathcal H(b)$. As the parameter approaches $b_c$ from below, the persistence of this dominant cyclic feature decreases rapidly. Immediately beyond the reference transition, the post-transient trajectory converges to the stable equilibrium and the finite one-dimensional persistence collapses to values close to zero.

The qualitative transition is preserved across the tested noise amplitudes. As in the Lorenz experiment, these calculations use one reproducible realization for each parameter and noise-level configuration, so the result is interpreted as stability across the tested noise levels and realizations rather than as a general statistical robustness statement.

\begin{figure}[!ht]
\centering
\includegraphics[width=0.75\textwidth]{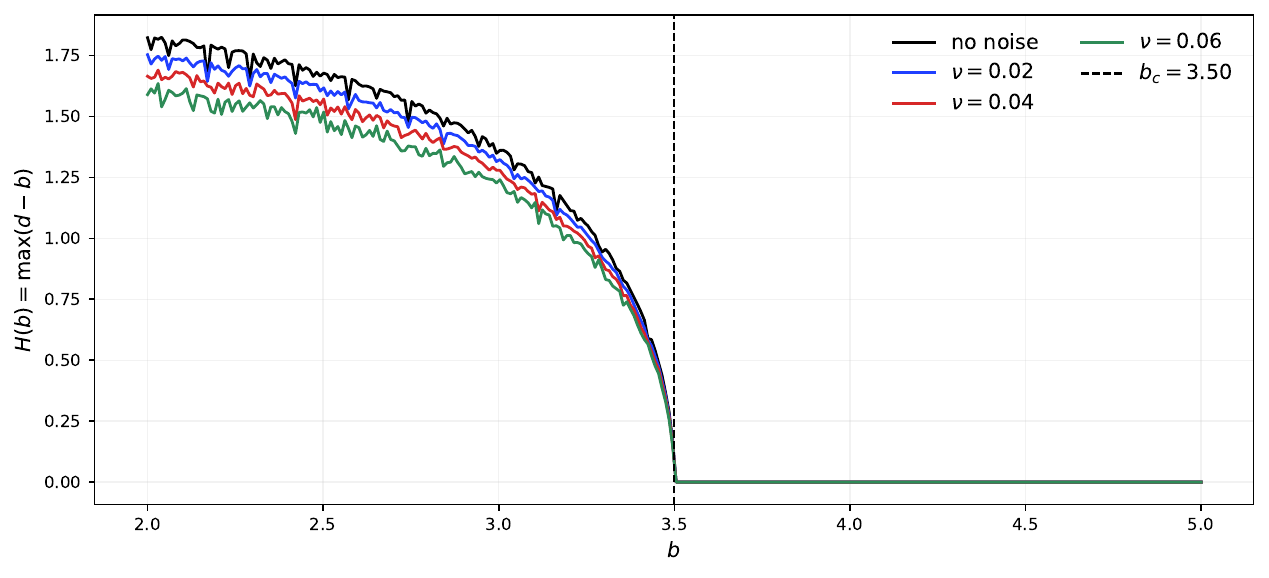}
\caption{Dominant one-dimensional persistence functional
$\mathcal H(b)=\max(d-b)$ for the BZ system under the tested relative Gaussian-noise levels $\nu\in\{0,0.02,0.04,0.06\}$. The vertical dashed line indicates the reference value $b_c=3.5$. The dominant cyclic persistence decreases as the attracting oscillatory branch contracts toward the transition and becomes negligible on the post-transition equilibrium side.}
\label{fig:bz_H_noise}
\end{figure}

Applying Definition~\ref{def:estimator} to the final BZ simulations gives $\widehat b_c^{\,\mathrm{TDA}}=3.474916$ for each of the tested noise levels $\nu\in\{0,0.02,0.04,0.06\}$.
Relative to the reference value $b_c=3.5$, the absolute error is $0.025084$ corresponding to a relative localization error of approximately $0.7167\%$. This result differs substantially from the preliminary estimate obtained with a much shorter integration horizon. The corrected value follows after controlling the slow near-critical transient described above, showing that the earlier large displacement was primarily a finite-transient artifact rather than evidence of an intrinsic bias of the topological estimator.

Figure~\ref{fig:bz_estimator} also illustrates why the absolute derivative in Definition~\ref{def:estimator} is useful for both orientations of the transition. In Case~A the dominant cyclic persistence is born as the control parameter increases, whereas here it disappears. The estimator depends on the magnitude of the topological change and is therefore insensitive to the sign of this orientation.

\begin{figure}[!ht]
\centering
\includegraphics[width=0.7\textwidth]{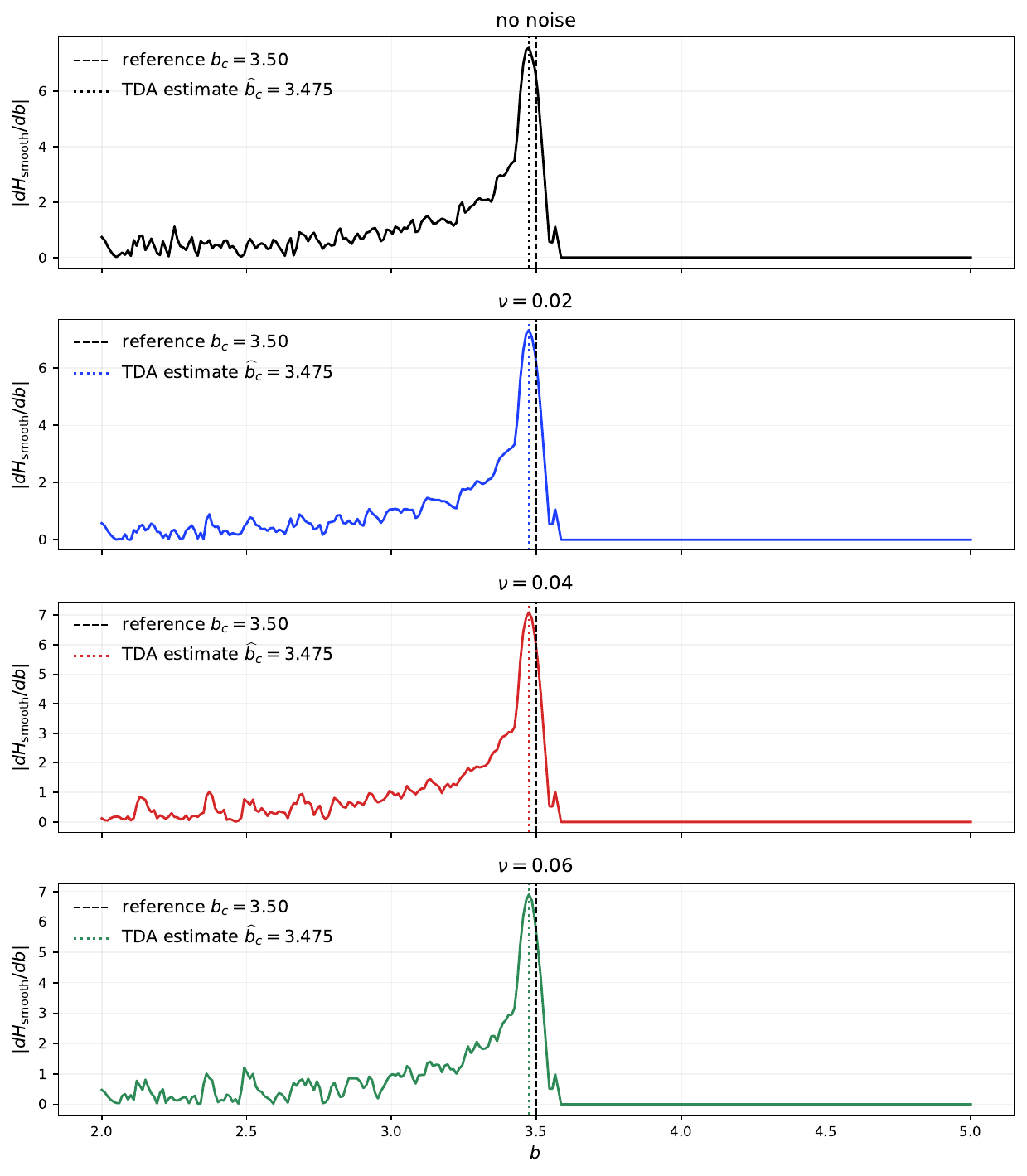}
\caption{Derivative-based localization of the BZ transition for the tested noise levels. The estimator selects the interior parameter value maximizing $|\widehat{\mathcal H}'(b)|$ after application of the common Savitzky--Golay smoothing and boundary-exclusion rules. All tested noise configurations yield $\widehat b_c^{\,\mathrm{TDA}}=3.474916$, close to the reference value $b_c=3.5$.
}
\label{fig:bz_estimator}
\end{figure}

To summarize the numerical localization results across the three systems, Table~\ref{tab:estimator} compares the final derivative-based estimates with their corresponding reference critical parameters. For the noisy experiments, each tested noise level is reported separately in order to make explicit the observed variation of the estimator. The errors should be interpreted as finite-resolution localization errors for the numerical procedure rather than as statistical estimation uncertainties.

\begin{table}[!ht]
\centering
\caption{Final performance of the derivative-based topological estimator. Absolute and relative errors are computed with respect to the reference critical parameter. The relative error is not defined for the Hopf normal form because its reference value is zero. }
\label{tab:estimator}
\scriptsize
\begin{tabular}{lcccc}
\toprule
System & Noise $\nu$ & Reference &
TDA estimate & Relative error \\
\midrule
Hopf   & 0    & $0$         & $0.016722$  & -- \\
\midrule
Lorenz & 0    & $24.736842$ & $24.648829$ & $0.3558\%$ \\
Lorenz & 0.04 & $24.736842$ & $24.615385$ & $0.4910\%$ \\
Lorenz & 0.10 & $24.736842$ & $24.615385$ & $0.4910\%$ \\
\midrule
BZ     & 0    & $3.5$       & $3.474916$  & $0.7167\%$ \\
BZ     & 0.02 & $3.5$       & $3.474916$  & $0.7167\%$ \\
BZ     & 0.04 & $3.5$       & $3.474916$  & $0.7167\%$ \\
BZ     & 0.06 & $3.5$       & $3.474916$  & $0.7167\%$ \\
\bottomrule
\end{tabular}
\end{table}

\paragraph{Sign of the finite-resolution displacement.}
For the Hopf normal form, the derivative-based estimate lies on the periodic side of the transition, $\widehat{\mu}^{\,\mathrm{TDA}}_c>\mu_c$. This is consistent with Proposition~\ref{prop:consistency}; the one-sided birth law places the singular onset of $|\mathcal H_\infty'(\mu)|$ strictly on the side $d(\mu)>0$, so regularizing this one-sided singularity on a finite parameter grid is expected to displace the numerical maximizer toward $d(\mu)>0$, i.e., toward the periodic side. This argument is specific to the supercritical attracting-cycle geometry of Proposition~\ref{prop:consistency} and does not extend to the Lorenz and BZ experiments: the classical Lorenz Hopf bifurcation is subcritical and therefore lies outside the hypotheses of Theorem~\ref{thm:hopf}, while the BZ criterion is evaluated under a reversed parameter orientation in which the dominant class disappears, rather than appears, as the parameter increases. Correspondingly, in both of these cases the estimated values fall below the reference parameter, $\widehat{\rho}^{\,\mathrm{TDA}}_c<\rho_c$ and $\widehat b_c^{\,\mathrm{TDA}}<b_c$. Since no one-sided birth law analogous to Proposition~\ref{prop:consistency} is established for the Lorenz or BZ systems, the sign of the finite-resolution displacement in these two cases is not predicted by the present theory and should be regarded as an empirical feature of the numerical procedure rather than a systematic bias with a known direction.

\subsection{Descriptive comparison with the largest Lyapunov exponent} \label{subsec:lyapunov_comparison}

The dominant topological functional and the largest Lyapunov exponent describe different aspects of a dynamical system. The quantity $\mathcal H$ is a geometric descriptor of the reconstructed point cloud, whereas Lyapunov exponents quantify asymptotic rates of infinitesimal separation or contraction along trajectories. We therefore do not expect a universal functional relationship between these quantities.

For a descriptive comparison, the largest Lyapunov exponent was evaluated along the same parameter grids used for the noise-free topological sweeps. The correlations reported below use the raw functional $\mathcal H$, without smoothing, at the same $M=300$ parameter values for each system and therefore require no interpolation. Pearson and Spearman coefficients are used only to summarize co-variation along the parameter sweeps. The sampled parameter values are not independent random replicates, so conventional iid-based $p$-values are not reported. For the Hopf normal form, the largest Lyapunov exponent was evaluated analytically. For the Lorenz and BZ systems, it was computed along the same noise-free parameter grids by integrating the corresponding variational equations with periodic renormalization of the tangent vector. Thus, the topological and dynamical quantities compared below are evaluated on the same $M=300$ parameter values.

These correlations are computed from the final, transient-controlled simulations described above, in particular the corrected long integration horizon adopted for the BZ reaction (Eqs.~\eqref{eq:bz_relaxation_rate}--\eqref{eq:bz_relaxation_time}). Shorter, insufficiently relaxed simulations near the critical parameter would be expected to yield different, and in general lower, correlation values, since an unresolved transient biases both the reconstructed topology and the estimated Lyapunov exponent; the BZ coefficients reported below should accordingly be read as specific to the corrected, transient-controlled pipeline rather than as an intrinsic property of the system independent of the simulation horizon.

\begin{table}[!ht]
\centering
\caption{ Descriptive Pearson and Spearman correlation coefficients between the raw dominant topological functional $\mathcal H$ and the largest Lyapunov exponent along the noise-free parameter sweeps. Each comparison uses the same $M=300$ parameter values. The coefficients are reported as descriptive measures of co-variation and are not interpreted as inferential tests. }
\label{tab:correlation}
\begin{tabular}{lcc}
\toprule
System & Pearson $r$ & Spearman $\rho_s$ \\
\midrule
Hopf normal form & 0.692938 & 0.884010 \\
Lorenz           & 0.968735 & 0.823589 \\
BZ reaction      & 0.724658 & 0.789635 \\
\bottomrule
\end{tabular}
\end{table}

The correlations in Table~\ref{tab:correlation} should not be interpreted as evidence that persistent homology and Lyapunov exponents measure the same dynamical quantity. In the Hopf normal form, for example, the asymptotic largest Lyapunov exponent is
\begin{equation*}
\lambda_1(\mu) =
\begin{cases}
\mu, & \mu<0,\\[1mm]
0,   & \mu\geq0,
\end{cases}
\end{equation*}

because the attracting periodic orbit has a neutral direction tangent to the flow. Thus, on the periodic side $\lambda_1$ remains zero while $\mathcal H(\mu)$ continues to increase with the geometric size of the reconstructed cycle. The relatively large rank association across the complete sweep therefore does not imply a direct functional dependence between $\mathcal H$ and the largest Lyapunov exponent.

For the Hopf normal form, a more direct dynamical relation is obtained by considering the \emph{transverse} Floquet exponent of the attracting periodic orbit.

\begin{proposition}[Normal-form relation with the transverse Floquet exponent]
\label{prop:normal_form_lyapunov}
Consider the supercritical Hopf normal form of Case~A,
\begin{equation*}
\dot r=\mu r-r^3, \qquad \dot\theta=\omega,
\end{equation*}
with $\mu>0$. The attracting periodic orbit has radius $r_*(\mu)=\sqrt{\mu}$, and its transverse Floquet exponent is

\begin{equation}\label{eq:transverse_floquet}
\lambda_\perp(\mu)=-2\mu.
\end{equation}

For a fixed scalar observation, fixed delay parameters, and corresponding sampling phases, the ideal delay reconstruction scales linearly with $r_*(\mu)$. Consequently, there exists a constant $\kappa>0$, determined by the unit-radius reconstruction and its sampling, such that

\begin{equation}\label{eq:H_normal_form_scaling}
 \mathcal H_\infty(\mu) = \kappa\sqrt{\mu}.
\end{equation}

Therefore,

\begin{equation} \label{eq:H_floquet_relation}
 \mathcal H_\infty(\mu)^2 = -\frac{\kappa^2}{2}\lambda_\perp(\mu).
\end{equation}
\end{proposition}

\begin{proof}
For $\mu>0$, the attracting periodic orbit satisfies $r_*(\mu)=\sqrt{\mu}$. Linearizing the radial equation around $r_*$ gives
\begin{equation*}
\lambda_\perp= \left. \frac{\partial}{\partial r} \left(\mu r-r^3\right) \right|_{r=r_*}= \mu-3r_*^2 = -2\mu.
\end{equation*}

For a fixed scalar observation and fixed delay parameters, the ideal delay-coordinate reconstruction at radius $r_*(\mu)$ is a Euclidean rescaling by $r_*(\mu)$ of the corresponding unit-radius reconstruction. Euclidean rescaling by a factor $c>0$ multiplies all pairwise distances, and hence all Vietoris--Rips birth and death scales and their persistence lifetimes, by the same factor. Therefore, $\mathcal H_\infty(\mu) = \kappa r_*(\mu) = \kappa\sqrt{\mu}$ for a constant $\kappa>0$, establishing \eqref{eq:H_normal_form_scaling}. Combining this identity with \eqref{eq:transverse_floquet} yields \eqref{eq:H_floquet_relation}.
\end{proof}

The relation predicted by Proposition~\ref{prop:normal_form_lyapunov} is also observed numerically in the final Hopf parameter sweep. On the periodic side, a linear fit of $\mathcal H(\mu)^2$ against $-\lambda_\perp(\mu)=2\mu$ gives

\begin{equation}\label{eq:empirical_H_floquet}
 \mathcal H(\mu)^2  = 1.483533\,[-\lambda_\perp(\mu)] + 1.24\times10^{-6}, \qquad R^2=0.999999999978.
\end{equation}

This is consistent with the independently estimated birth-law coefficient in Figure~\ref{fig:prop1_birthlaw}, since

\begin{equation*}
 \frac{\widehat{\kappa}^{\,2}}{2} = \frac{(1.7225)^2}{2} \simeq 1.4835.
\end{equation*}

Thus, for the canonical normal form, the finite-data calculation reproduces with high numerical accuracy the proportionality predicted by Eq.~\eqref{eq:H_floquet_relation}.

The comparison has a different interpretation for the other two systems. For the Lorenz system, the high Pearson coefficient $r=0.968735$ indicates strong co-variation between the reconstructed topological geometry and the largest Lyapunov exponent across the parameter sweep. However, the classical Lorenz Hopf bifurcation is subcritical and the experiment does not satisfy the attracting-periodic-orbit assumptions of Theorem~\ref{thm:hopf}. The observed association is therefore descriptive and reflects the simultaneous response of two different quantities to a major dynamical reorganization; it should not be interpreted as validation of a universal persistence--Lyapunov relationship.

For the BZ system, the corresponding coefficients, $r=0.724658$ and $\rho_s=0.789635$, also indicate substantial co-variation along the parameter sweep. Nevertheless, the largest Lyapunov exponent remains non-positive over the considered parameter range while the dominant persistence changes markedly as the oscillatory branch contracts and disappears. This again illustrates that $\mathcal H$ and the largest Lyapunov exponent characterize different aspects of the dynamics: the former describes global geometry in the reconstructed point cloud, whereas the latter describes local asymptotic stability.

Taken together, these comparisons support a complementary rather than equivalent interpretation of the topological and Lyapunov-based descriptors. The exact relation \eqref{eq:H_floquet_relation} is specific to the scaling geometry of the Hopf normal form under a fixed delay reconstruction and is not asserted for the Lorenz or BZ systems.

\section{Conclusions and Future Directions}
\label{sec:discussion}

\subsection*{Summary of findings}

In this work, we developed a topological framework for localizing Hopf-type dynamical transitions directly from scalar time series using persistent homology of fixed delay-coordinate reconstructions. The central observable is the dominant one-dimensional persistence $\mathcal H(\theta)$, which measures the prominence of the most persistent cyclic feature in the reconstructed point cloud.

The theoretical analysis establishes pointwise persistence bounds for the supercritical attracting cycle setting. In particular, Theorem~\ref{thm:hopf} relates detectability of the reconstructed cycle to its reach, finite sampling density, and finite-data perturbations, while Corollary~\ref{cor:finite_resolution} makes explicit that a uniform positive separation can only be expected on parameter regions bounded away from the critical value, consistent with the absence of a uniform threshold near $\mu_c$ noted in Remark~\ref{rem:finite_resolution}. Corollary~\ref{cor:gaussian_noise} provides a finite-sample probabilistic connection between these deterministic perturbation bounds and Gaussian observational noise.

A derivative-based estimator was then introduced to localize the strongest change of the sampled topological functional. The numerical experiments show that its output must be interpreted at finite resolution rather than as an exact finite-grid recovery theorem. For the Hopf normal form, the baseline computation gives $\widehat{\mu}^{\,\mathrm{TDA}}_c= 0.016722$ (\eqref{eq:hopf_final_estimate}), relative to the exact value $\mu_c=0$. The corresponding birth-law analysis yields $\widehat{\alpha}=0.5000$, recovering numerically the square-root scaling predicted by the normal-form geometry. The finite-resolution study further shows that the location of the derivative maximum depends materially on the smoothing bandwidth, while remaining unchanged across the tested temporal sampling frequencies when the physical delay and observation horizon are kept fixed.

The Lorenz experiment provides a complementary empirical stress test. For the classical parameters, the Hopf bifurcation of the nonzero equilibria is subcritical and therefore lies outside the attracting-periodic-orbit assumptions of Theorem~\ref{thm:hopf}. Nevertheless, the strongest change of the topological functional is localized close to the reference value $\rho_c=24.736842\ldots$, with relative errors below $0.5\%$ for all tested noise configurations. This result is interpreted as empirical detection of a pronounced geometric reorganization rather than as a direct validation of the theorem.

For the BZ system, the attracting oscillatory branch lies on the $b<b_c$ side and the dominant cyclic persistence disappears as the parameter increases through the transition. After explicitly controlling the slow near-critical transient, the final estimator gives $\widehat b_c^{\,\mathrm{TDA}}=3.474916$, compared with the reference value $b_c=3.5$, corresponding to a relative localization error of approximately $0.7167\%$. The large displacement observed in preliminary shorter simulations was therefore not retained as an intrinsic bias of the estimator; it was primarily associated with insufficient transient removal.

Finally, the comparison with the largest Lyapunov exponent presented in Section~\ref{subsec:lyapunov_comparison} indicates that the topological and dynamical quantities should be viewed as complementary rather than equivalent descriptors. For the Hopf normal form, the more direct theoretical connection is with the transverse Floquet exponent. Proposition~\ref{prop:normal_form_lyapunov} gives $ \mathcal H_\infty(\mu)^2=-\frac{\kappa^2}{2}\lambda_\perp(\mu)$ (\eqref{eq:H_floquet_relation}), and the finite-data calculation reproduces this proportionality with high numerical accuracy (\eqref{eq:empirical_H_floquet}). No analogous universal relationship is asserted for the Lorenz or BZ systems.

\subsection*{Impact and relevance}
The results support the use of persistent homology as a quantitative descriptor of geometric reorganizations in reconstructed nonlinear dynamics. The contribution is not the first use of persistence-based quantities for detecting dynamical transitions: previous approaches have already employed quantitative distances between persistence diagrams and other topological summaries \cite{BerwaldGidea2014,Guzel2022,Tymochko2020}. The distinction of the present framework lies in combining a particularly simple scalar $H_1$ observable with explicit geometric persistence bounds and an operational derivative-based localization rule.

An important feature of the proposed observable is its geometric interpretability. In the canonical supercritical Hopf setting, $\mathcal H$ directly reflects the growth of a reconstructed cyclic attractor, and the numerical square-root law connects the persistence scale to the amplitude law of the bifurcation. At the same time, the BZ experiment shows that the derivative-based procedure can be applied when the same cyclic feature disappears rather than appears as the control parameter is increased.

The computational procedure itself requires only an ordered family of scalar time series and fixed reconstruction parameters; it does not require evaluation of a Jacobian, eigenvalues, or Lyapunov exponents. The governing equations were used in the present synthetic experiments to generate data and provide reference critical values, not as inputs to the topological estimator.

\subsection*{Scope and limitations}

The theoretical and numerical results should be interpreted within several aspects of the scope of the present study. First, Theorem~\ref{thm:hopf} is formulated specifically for a supercritical Hopf bifurcation in which post-transient trajectories sample an attracting periodic orbit. The classical Lorenz experiment lies outside this setting because its Hopf bifurcation is subcritical and is therefore included as an empirical stress test rather than as a direct validation of the theorem. The method also depends on the geometry induced by the delay reconstruction. Although the embedding parameters are held fixed within each parameter sweep and their sensitivity is examined in the Hopf experiment, the value $m=2$ used here should not be interpreted as a general guarantee that the sufficient conditions of Takens' embedding theorem are satisfied, particularly for the Lorenz attractor.

A second limitation concerns numerical localization. The persistence bounds of Theorem~\ref{thm:hopf} establish geometric detectability, whereas Definition~\ref{def:estimator} identifies the maximum derivative of a smoothed functional evaluated on a finite parameter grid. The present analysis therefore does not establish a finite-sample convergence rate or confidence interval for $\widehat{\theta}^{\,\mathrm{TDA}}_c$. The Hopf sensitivity study shows that smoothing bandwidth can materially affect the location of the numerical maximum. Moreover, the Hopf and BZ experiments illustrate that slow near-critical transients must be adequately removed before the reconstructed topology can be interpreted as representative of the asymptotic regime; as noted in Section~\ref{sec:experiments}, no analogous transient diagnostic was performed for the Lorenz system, which is a further reason to read its localization result descriptively rather than as a validation of Theorem~\ref{thm:hopf}.

The Gaussian-noise experiments should also be interpreted as reproducible sensitivity tests rather than as a complete statistical robustness analysis. One realization was used for each parameter and noise-level configuration, with paired Gaussian realizations across positive noise amplitudes at a fixed parameter value. Consequently, the observed stability across the tested configurations does not by itself characterize the sampling distribution or uncertainty of the estimator under repeated observational noise.

Finally, the present numerical validation is based on synthetic dynamical systems for which parameterized families of time series and reference critical values are available. Application to experimental data will require addressing additional effects such as irregular sampling, nonstationarity, unknown transient duration, and parameter uncertainty. Persistent-homology computations may also become more demanding for substantially larger reconstructed point clouds, so statistical and computational scalability remain relevant topics for future work.

\subsection*{Future directions}

Several extensions follow naturally from these limitations. A first direction is the development of statistical uncertainty quantification for the critical-parameter estimator, including repeated-noise experiments, bootstrap or subsampling strategies adapted to dependent time-series data, and finite-sample confidence statements for topological transition localization.

A second direction is a more systematic theory for reconstruction parameters. In particular, it would be useful to quantify how delay, embedding dimension, sampling rate, and point-cloud subsampling affect the reach and persistence margin appearing in Theorem~\ref{thm:hopf}. Such results could connect practical delay-coordinate selection more directly with the geometric detectability conditions of the theory.

A third direction concerns the localization estimator itself. The continuous birth-law analysis explains why the derivative becomes singular in the canonical supercritical case, but a complete theory should characterize the combined effects of parameter discretization, finite observation time, persistence estimation, and smoothing on $\widehat{\theta}^{\,\mathrm{TDA}}_c$. Adaptive or multiscale regularization strategies may reduce the dependence on a fixed Savitzky--Golay bandwidth.

Another important direction is the automatic diagnosis of transient convergence. The BZ experiment demonstrates that near-critical slowing can strongly contaminate the reconstructed topology if the retained trajectory has not reached the relevant asymptotic regime. Data-driven criteria for deciding when a trajectory is sufficiently stationary for topological analysis would substantially improve applicability when the equations and linear relaxation rates are unavailable; extending such a diagnostic to systems such as Lorenz, for which it was not performed here, is a natural first step.

The framework could also be extended to other bifurcation mechanisms, including period-doubling, torus, and global bifurcations, for which different persistent-homology signatures or homological dimensions may be more informative. Finally, application to experimental biological, chemical, physiological, or engineering time series is an essential next step for determining how the proposed geometric criterion behaves under realistic sampling, measurement noise, nonstationarity, and model uncertainty.

\bibliography{Bibliography}
\bibliographystyle{acm}

\end{document}